\documentclass[12pt,reqno]{amsart}
\usepackage{amsmath,amsopn,amssymb,amsthm,multicol}
\usepackage{hyperref}
\usepackage{graphicx}

\DeclareMathOperator{\Ad}{Ad}

\DeclareMathOperator{\Ric}{Ric}

\newcommand{\fr}{\mathfrak}

\newcommand{\bb}{\mathbb}

\DeclareMathOperator{\SO}{SO}

\DeclareMathOperator{\Sp}{Sp}
\DeclareMathOperator{\SU}{SU}

 \newtheorem{lemma} {Lemma} [section]
\newtheorem{theorem}[lemma]{Theorem} 
 
\newtheorem{prop} [lemma]{Proposition}

\begin{document}

\title{New Einstein metrics on the Lie group $\SO (n)$ which are not naturally reductive} 
\author{Andreas Arvanitoyeorgos, Yusuke Sakane, and Marina Statha}
\address{University of Patras, Department of Mathematics, GR-26500 Patras, Greece}
\email{arvanito@math.upatras.gr}
 \address{Osaka University, Department of Pure and Applied Mathematics, Graduate School of Information Science and Technology, Toyonaka, 
Osaka 560-0043, Japan}
 \email{sakane@math.sci.osaka-u.ac.jp}
\address{University of Patras, Department of Mathematics, GR-26500 Patras, Greece}
\email{statha@master.math.upatras.gr} 
\medskip

   \begin{abstract}
We obtain new invariant Einstein metrics  on the compact Lie groups
 $\SO(n)$ ($n \geq 7$) which are not naturally reductive.  This is achieved by  imposing certain symmetry assumptions in the set of all 
 left-invariant metrics on $\SO(n)$ and by computing the Ricci tensor for such metrics.  The Einstein metrics are obtained as solutions of systems polynomial equations, which we manipulate by symbolic computations  using Gr\"obner bases.
  

   \end{abstract}

\maketitle


 \section{Introduction}
\markboth{Andreas Arvanitoyeorgos, Yusuke Sakane and Marina Statha}{Einstein metrics on $\SO (n)$ which are not naturally reductive}

A Riemannian manifold $(M, g)$ is called Einstein if it has constant Ricci curvature, i.e. $\Ric_{g}=\lambda\cdot g$ for some $\lambda\in\bb{R}$. 
  For results on Einstein manifolds before 1987 we refer to the book by A. Besse \cite{Be}. 
  The two  articles \cite{W1}, \cite{W2} of M. Wang  contain results up to 1999 and 2013 respectively.
  General existence results are difficult to obtain and some methods are described
  in \cite{Bom}, \cite{BWZ} and \cite{WZ}.  
For homogeneous spaces $G/K$ the problem is to find and classify all $G$-invariant Einstein metrics.
The problem is even more difficult for the case of a Lie group, where we need to find (or  prove existence) of left-invariant Einstein metrics.
Even for the compact Lie groups $\SU(3)$ and $\SU(2)\times\SU(2)$ the number of left-invariant Einstein metrics is still unknown. 

    In the present paper we investigate left-invariant Einstein metrics on the compact Lie group
    $\SO(n)$.
 It is known that a compact and semisimple Lie group equipped with a bi-invariant metric is Einstein.
In the work \cite{DZ}, J.E. D'Atri and W. Ziller  found a large number of left-invariant Einstein metrics, which are naturally reductive,  on the compact Lie groups
  $\SU(n), \SO(n)$ and $\Sp(n)$.
  In the same article they posed the question whether there exist left-invariant Einstein metrics on compact Lie groups which are not naturally reductive.
  
  Some contributions to this problem are the following:
  In \cite{M} K. Mori obtained non naturally reductive Einstein metrics on the Lie group $\SU(n)$ for $n\ge 6$, and
  in \cite{AMS} the first two  authors and K. Mori  proved existence of non naturally reductive Einstein metrics on the compact Lie groups $\SO(n)$ $(n\ge 11)$, $\Sp(n)$ $(n\ge 3)$, $E_6, E_7$ and $E_8$. 
   In \cite{CL} Z. Chen and K. Liang found three naturally reductive and one non naturally reductive Einstein metric on the compact Lie group $F_4$.  Also, in \cite{ASS2} the authors  obtained new left-invariant Einstein metrics on the symplectic group $\Sp(n)\ (n\ge 3)$, and in
   \cite{CS} I. Chrysikos and the second author obtained non naturally reductive Einstein metrics on exceptional Lie groups.
   We also mention the works  \cite{GLP}, \cite{Po}  by G.W. Gibbons, H. L\"u and C. N. Pope  where they discussed left-invariant Einstein metrics on the Lie groups  $\SO(n)$,  $G_2$ and $\SU(3)$ which are however naturally reductive,
  as well as  \cite{Mu}  by A. H. Mujtaba, who obtained certain classes of left-invariant metrics on 
  $\SU(n)$, which were
  previously found in \cite{J2} and \cite{DZ}.
  
  The aim of the present work is to obtain left-invariant Einstein metrics on the compact Lie groups $\SO(n)$ $(n\ge 7)$ which are not naturally reductive. The Einstein metrics obtained here are different from the ones obtained in \cite{AMS}.  The idea behind our approach is to consider an appropriate subgroup $K$ and a corresponding homogeneous space $G/K$ whose isotropy representation decomposes into $\Ad(K)$-irreducible and non equivalent summands.  Then the tangent space $\fr{g}$ of $G$ decomposes, via the submersion $G\to G/K$ with fiber $K$, into a direct sum of non equivalent $\Ad(K)$-modules.
  By taking into account the diffeomorphism 
  $G/\{e\}\cong (G\times K)/\mbox{diag}(K)$ we consider left-invariant metrics on $G$ which are determined by diagonal $\Ad(K)$-invariant scalar products on $\fr{g}$, which in turn enables us to use well known formulas for the Ricci curvature 
  (e.g. \cite[Corollary 7.38, p. 184]{Be}, \cite[Lemma 1.1, p. 52]{PS}).
  
More precisely, for the group $\SO(n)$ we write  $n=k_1+k_2+k_3$,  ($k_1, k_2, k_3$  positive integers),  and we
   consider the closed subgroup $K=\SO(k_1)\times\SO(k_2)\times\SO(k_3)$.  This determines the homogeneous space $G/K = \SO(k_1+k_2+k_3)/(\SO(k_1)\times\SO(k_2)\times\SO(k_3))$, which is an example of a {\it generalized Wallach space} according to \cite{NRS}.  For $k_1\ge k_2\ge k_3\ge 2$ with $k_i\ne k_j$ the isotropy representation 
   $\fr{m}=\fr{m}_{12}\oplus\fr{m}_{13}\oplus\fr{m}_{23}$ 
   of $G/K$ does not contain equivalent summands, thus we consider left-invariant metrics determined by  $\Ad(\SO(k_1)\times\SO(k_2)\times\SO(k_3))$-invariant inner products of the form
$$
\begin{array}{lll}
 \langle \  ,\  \rangle &=&  x_1 \, (-B) |_{\fr{so}(k_1)}+ x_2 \, (-B) |_{ \fr{so}(k_2)}+  x_3 \, (-B) |_{ \fr{so}(k_3)} 
 \\ &  & + x_{12} \,  (-B) |_{ \fr{m}_{12}}+ x_{13} \,  (-B) |_{ \fr{m}_{13}} + x_{23} \,  (-B) |_{ \fr{m}_{23}} 
 \end{array},\quad x_i, x_{ij} >0.
 $$
 If   $k_3=1$, $k_1\ge k_2\ge 2$ we omit the variable $x_3$,
 and if   $k_2=k_3=1$, $k_1\ge 3$ we omit the variables $x_2$ and $x_3$.
 In the last case the submodules $\fr{m}_{12}$ and $\fr{m}_{13}$ are equivalent and we treat this separately in Section 7.
  
 By using the main theorem of D'Atri and Ziller (\cite{DZ}) we obtain conditions on the positive variables $x_1, x_2, x_3, x_{12}, x_{13}$ and $x_{23}$ so that
 the above metric is naturally reductive.
 For all possible partitions of $n=k_1+k_2+k_3$ we write the components of the Ricci tensor and then  use methods of symbolic computation to prove existence of positive solutions to systems of algebraic equations obtained by the Einstein equation.
For the case of the Lie groups $\SO(5)$ and  $\SO(6)$ our method gives only naturally reductive metrics.
  
  \smallskip
  \noindent
The main result  is the following:

\begin{theorem}\label{main}  The compact simple Lie groups $\SO(n)$ $(n \geq 7) $  admit left-invariant Einstein metrics which are not naturally reductive. 
\end{theorem}

The paper is organized as follows: In Section 2 we recall a formula for the Ricci tensor of homogeneous spaces given in \cite{PS}.  In Section 3 we describe the left-invariant metrics which will be considered in this work, and in Section 4 we give conditions under which such left-invariant metrics on $\SO(n)$ are naturally reductive with respect to $\SO(n)\times L$, for some closed subgroup $L$ of $\SO(n)$.
In Section 5 we obtain explicit formulas for the Ricci tensor of left-invariant metrics.  
In Section 6 we investigate the solutions of the Einstein equation by using Gr\"obner bases.
Then  Theorem \ref{main} follows from Propositions 
\ref{prop1}, \ref{prop2} and \ref{prop3}.  Finally,
in Section 7 we prove that the compact Lie groups $\SO(n)$  admit only naturally reductive $\Ad(\SO(n-2))$-invariant Einstein metrics  (here $k_1=n-2, k_2=k_3=1$).
This case is of special interest, because the decomposition  $\fr{so}(n)=\fr{so}(k_1)\oplus \fr{m}_{12}\oplus\fr{m}_{13}\oplus\fr{m}_{23}$  of the tangent space of $\SO(n)$
contains equivalent
$\Ad(\SO(n-2))$-modules.  Hence, we need to confirm that for the
 $\Ad(\SO(n-2))$-invariant metrics under consideration
the Ricci tensor is diagonal.

\medskip
{\bf Acknowledgements.} 
The work was supported by Grant $\# E.037$ from the Research Committee of the University of Patras
(Programme K. Karatheodori) and JSPS KAKENHI Grant Number 25400071.
  It was completed while the first author was on sabbatical leave at Tufts University, University of Athens and Osaka University during 2014-15.

\section{The Ricci tensor for reductive homogeneous spaces}
In this section we recall an expression for the Ricci tensor for an $G$-invariant Riemannian
metric on a reductive homogeneous space whose isotropy representation
is decomposed into a sum of non equivalent irreducible summands.

Let $G$ be a compact semisimple Lie group, $K$ a connected closed subgroup of $G$  and  
let  $\frak g$ and $\fr{k}$  be  the corresponding Lie algebras. 
The Killing form $B$ of $\frak g$ is negative definite, so we can define an $\mbox{Ad}(G)$-invariant inner product $-B$ on 
  $\frak g$. 
Let $\frak g$ = $\frak k \oplus
\frak m$ be a reductive decomposition of $\frak g$ with respect to $-B$ so that $\left[\,\frak k,\, \frak m\,\right] \subset \frak m$ and
$\frak m\cong T_o(G/K)$.
 We assume that $ {\frak m} $ admits a decomposition into mutually non equivalent irreducible $\mbox{Ad}(K)$-modules as follows: \ 
\begin{equation}\label{iso}
{\frak m} = {\frak m}_1 \oplus \cdots \oplus {\frak m}_q.
\end{equation} 
Then any $G$-invariant metric on $G/K$ can be expressed as  
\begin{eqnarray}
 \langle\  , \  \rangle  =  
x_1   (-B)|_{\mbox{\footnotesize$ \frak m$}_1} + \cdots + 
 x_q   (-B)|_{\mbox{\footnotesize$ \frak m$}_q},  \label{eq2}
\end{eqnarray}
for positive real numbers $(x_1, \dots, x_q)\in\bb{R}^{q}_{+}$.  Note that  $G$-invariant symmetric covariant 2-tensors on $G/K$ are 
of the same form as the Riemannian metrics (although they  are not necessarilly  positive definite).  
 In particular, the Ricci tensor $r$ of a $G$-invariant Riemannian metric on $G/K$ is of the same form as (\ref{eq2}), that is 
 \[
 r=y_1 (-B)|_{\mbox{\footnotesize$ \frak m$}_1}  + \cdots + y_{q} (-B)|_{\mbox{\footnotesize$ \frak m$}_q} ,
 \]
 for some real numbers $y_1, \ldots, y_q$.

Let $\lbrace e_{\alpha} \rbrace$ be a $(-B)$-orthonormal basis 
adapted to the decomposition of $\frak m$,    i.e. 
$e_{\alpha} \in {\frak m}_i$ for some $i$, and
$\alpha < \beta$ if $i<j$. 
We put ${A^\gamma_{\alpha
\beta}}= -B \left(\left[e_{\alpha},e_{\beta}\right],e_{\gamma}\right)$ so that
$\left[e_{\alpha},e_{\beta}\right]
= \displaystyle{\sum_{\gamma}
A^\gamma_{\alpha \beta} e_{\gamma}}$ and set 
$\displaystyle{k \brack {ij}}=\sum (A^\gamma_{\alpha \beta})^2$, where the sum is
taken over all indices $\alpha, \beta, \gamma$ with $e_\alpha \in
{\frak m}_i,\ e_\beta \in {\frak m}_j,\ e_\gamma \in {\frak m}_k$ (cf. \cite{WZ}).  
Then the positive numbers $\displaystyle{k \brack {ij}}$ are independent of the 
$B$-orthonormal bases chosen for ${\frak m}_i, {\frak m}_j, {\frak m}_k$,
and 
$\displaystyle{k \brack {ij}}\ =\ \displaystyle{k \brack {ji}}\ =\ \displaystyle{j \brack {ki}}.  
 \label{eq3}
$

Let $ d_k= \dim{\frak m}_{k}$. Then we have the following:

\begin{lemma}\label{ric2}\textnormal{(\cite{PS})}
The components ${ r}_{1}, \dots, {r}_{q}$ 
of the Ricci tensor ${r}$ of the metric $ \langle  \,\,\, , \,\,\, \rangle $ of the
form {\em (\ref{eq2})} on $G/K$ are given by 
\begin{equation}
{r}_k = \frac{1}{2x_k}+\frac{1}{4d_k}\sum_{j,i}
\frac{x_k}{x_j x_i} {k \brack {ji}}
-\frac{1}{2d_k}\sum_{j,i}\frac{x_j}{x_k x_i} {j \brack {ki}}
 \quad (k= 1,\ \dots, q),    \label{eq51}
\end{equation}
where the sum is taken over $i, j =1,\dots, q$.
\end{lemma} 
Since by assumption the submodules $\fr{m}_{i}, \fr{m}_{j}$ in the decomposition (\ref{iso}) are mutually non equivalent for any $i\neq j$, it is $r(\fr{m}_{i}, \fr{m}_{j})=0$ whenever $i\neq j$. 
If $\fr{m}_i\cong\fr{m}_j$ for some $i\ne j$ then we need to check whether 
$r(\fr{m}_{i}, \fr{m}_{j})=0$. This is not an easy task in general.
Once the condition $r(\fr{m}_{i}, \fr{m}_{j})=0$ is confirmed we can use
  Lemma \ref{ric2}.   Then    $G$-invariant Einstein metrics on $M=G/K$ are exactly the positive real solutions $g=(x_1, \ldots, x_q)\in\bb{R}^{q}_{+}$  of the  polynomial system $\{r_1=\lambda, \ r_2=\lambda, \ \ldots, \ r_{q}=\lambda\}$, where $\lambda\in \bb{R}_{+}$ is the Einstein constant.

\section{A class of left-invariant metrics on $\SO(n) = \SO(k_1+k_2+k_3)$}

We will describe a decomposition of the tangent space of the Lie group $\SO (n)$ which will be convenient for our study.
 We consider the closed subgroup $K = \SO(k_1)\times\SO(k_2)\times\SO(k_3)$ of $G = \SO(k_1 + k_2 + k_3)$
 ($k_1\ge k_2\ge k_3\ge 2$), where the embedding of $K$ in $G$ is diagonal,
  and  the fibration
$$
\SO(k_1)\times\SO(k_2)\times\SO(k_3) \to  \SO(k_1 + k_2 + k_3) \to  \SO(k_1 + k_2 + k_3)/(\SO(k_1)\times\SO(k_2)\times\SO(k_3)).
$$
The base space of the above fibration is an example of a {\it generalized Wallach space} (cf. \cite{NRS}). 
  Then the tangent space $\fr{so}(k_1 + k_2 + k_3)$ of the orthogonal group $G = \SO(k_1 + k_2 + k_3)$ can be written as a direct sum of two $\Ad(K)$-invariant modules, the horizontal space $\fr{m}\cong T_{o}(G/K)$ and the vertical space $\fr{so}(k_1)\oplus\fr{so}(k_2)\oplus\fr{so}(k_3)$, i.e.
\begin{equation}\label{diaspasi}
\fr{so}(k_1 + k_2 + k_3) = \fr{so}(k_1)\oplus\fr{so}(k_2)\oplus\fr{so}(k_3)\oplus\fr{m}.
\end{equation}
The tangent space $\fr{m}$ of $G/K$ is given by $\fr{k}^{\perp}$ in $ \fr{g} = \fr{so}(k_1+ k_2+k_3)$ with respect to  $-B$.
 If we denote by $M(p,q)$ the set of all $p \times q$ matrices, then we see that 
  $ \fr{m}$ is given by 
  \begin{equation}
\fr{m}=  \left\{\begin{pmatrix}
 0 & {A}_{12} & {A}_{13}\\
 -{}^{t}_{}\!{A}_{12} & 0 & {A}_{23}\\
 -{}^{t}_{}\!{A}_{13} & -{}^{t}_{}\!{A}_{23} & 0 
 \end{pmatrix} \  \Big\vert \  {A}_{12} \in M(k_1, k_2), {A}_{13} \in M(k_1, k_3), {A}_{23} \in M(k_2, k_3) \right\} 
 \end{equation}
and we have 
   \begin{equation}\label{modules}
 \fr{m}_{12}= \begin{pmatrix}
 0 & {A}_{12} & 0\\
 -{}^{t}_{}\!{A}_{12} & 0 &0\\
0  & 0 & 0 
 \end{pmatrix},  \quad  
 \fr{m}_{13}= \begin{pmatrix}
 0 & 0 &{A}_{13}\\
0 & 0 &0\\
 -{}^{t}_{}\!{A}_{13}   & 0 & 0 
 \end{pmatrix},  \quad  
 \fr{m}_{23}= \begin{pmatrix}
 0 & 0 & 0\\
0 & 0 &{A}_{23}\\
0  &  -{}^{t}_{}\!{A}_{23} & 0 
 \end{pmatrix}.  
 \end{equation}
 Note that the action of $\Ad(k)$ ($k \in K$) on $ \fr{m}$ is given by 
 \begin{equation}
\Ad(k) \begin{pmatrix}
 0 & {A}_{12} & {A}_{13}\\
 -{}^{t}_{}\!{A}_{12} & 0 & {A}_{23}\\
 -{}^{t}_{}\!{A}_{13} & -{}^{t}_{}\!{A}_{23} & 0 
 \end{pmatrix}  = 
  \begin{pmatrix}
 0 &{}^t h_1 {A}_{12} h_2 & {}^t h_1{A}_{13} h_3\\
 -{}^{t}_{}h_2 {}^{t}_{}\!{A}_{12}h_1 & 0 & {}^t h_2{A}_{23} h_3\\
 -{}^{t}_{}h_3{}^{t}_{}\!{A}_{13} h_1 & -{}^{t}_{}h_3{}^{t}_{}\!{A}_{23} h_2& 0 
 \end{pmatrix}, 
  \end{equation}
 where $ \begin{pmatrix}
 h_1 & 0& 0\\
0 & h_2 &0\\
0  & 0 & h_3 
 \end{pmatrix} \in K$.  
 The subspaces  $\fr{m}_{12}$,  $\fr{m}_{13}$ and  $\fr{m}_{23}$  are  
irreducible $\Ad(K)$-submodules whose dimensions are
$\dim\fr{m}_{12}=k_1k_2$,  $\dim\fr{m}_{13}=k_1k_3$ and $\dim\fr{m}_{23}=k_2k_3$.
They are given as $(-B)$-orthogonal complements of
$\fr{so}(k_i)\oplus\fr{so}(k_j)$ in $\fr{so}(k_i+k_j)$ ($1\le i<j\le 3$), respectively.

Note that the irreducible submodules $\fr{m}_{ij}$ are mutually non equivalent,
 so any $G$-invariant metric on the base space $G/K$ is determined by an $\Ad(K)$-invariant scalar product
$
x_{12} (-B) |_{ \fr{m}_{12}}$  $+ x_{13}(-B) |_{ \fr{m}_{13}}+ x_{23}(-B) |_{ \fr{m}_{23}}.
$
We also set $\fr{m}_1=\fr{so}(k_1), \fr{m}_2=\fr{so}(k_2)$ and $\fr{m}_3=\fr{so}(k_3)$.
Therefore, decomposition (\ref{diaspasi}) of the tangent space of the orthogonal group $G= \SO(k_1+k_2+k_3)$ takes the form

\begin{equation}\label{decom_so(n)}
 \fr{so}(k_1+k_2+k_3) = \fr{m}_1\oplus \fr{m}_2\oplus \fr{m}_3\oplus  \fr{m}_{12}\oplus  \fr{m}_{13}\oplus  \fr{m}_{23}. 
\end{equation}

Let $E_{ab}$ denotes the $n\times n$ matrix with $1$ at the $(ab)$-entry and $0$ elsewhere.
Then the set
$\mathcal{B}=\{e_{ab}=E_{ab}-E_{ba}: 1\le a<b\le n\}$
constitutes a $(-B)$-orthonormal basis of $\fr{so}(n)$.
Note that $e_{ba}=-e_{ab}$, thus we have the following:

\begin{lemma}\label{brac}
If all four indices are distinct, then the Lie brakets in $\mathcal{B}$ are zero.
Otherwise,
$[e_{ab}, e_{bc}]=e_{ac}$, where $a,b,c$ are distinct.
\end{lemma}

By using Lemma \ref{brac} we obtain:

\begin{lemma}\label{brackets}  The submodules in the decomposition (\ref{decom_so(n)}) satisfy the following bracket relations:
\begin{center}
\begin{tabular}{lll}
$[ \fr{m}_1, \fr{m}_1] = \fr{m}_1,$  &   $[ \fr{m}_2, \fr{m}_2] = \fr{m}_2,$  & $[ \fr{m}_3, \fr{m}_3] = \fr{m}_3,$ \\
$[ \fr{m}_1, \fr{m}_{12}] = \fr{m}_{12},$   &   $[ \fr{m}_1, \fr{m}_{13}] = \fr{m}_{13},$  & 
$[\fr{m}_2, \fr{m}_{12}] =  \fr{m}_{12},$\\ 
$[ \fr{m}_2, \fr{m}_{23}] = \fr{m}_{23},$ & 
  $[ \fr{m}_3, \fr{m}_{13}] = \fr{m}_{13},$  & $[ \fr{m}_3, \fr{m}_{23}] = \fr{m}_{23},$ \\
$ [ \fr{m}_{12}, \fr{m}_{23}] = \fr{m}_{13},$   &   $[ \fr{m}_{13}, \fr{m}_{23}] = \fr{m}_{12},$  &
 $[\fr{m}_{12}, \fr{m}_{13}] = \fr{m}_{23}$\\ 
$[ \fr{m}_{12}, \fr{m}_{12}] = \fr{m}_{1}\oplus  \fr{m}_{2},$   &   $[ \fr{m}_{13}, \fr{m}_{13}] = \fr{m}_{1} \oplus \fr{m}_{3},$  &
 $[\fr{m}_{23}, \fr{m}_{23}] = \fr{m}_{2} \oplus \fr{m}_{3}.$ 
\end{tabular}
\end{center}  
  \end{lemma}
  
   Therefore,  we see that the only non zero symbols (up to permutation of indices) are 
  \begin{equation}\label{triplets}
{1 \brack {11}},  {2 \brack {22}},   {3 \brack {33}}, {(12) \brack {1(12)}},   {(13) \brack {1(13)}}, 
   {(12) \brack {2(12)}},    {(23) \brack {2(23)}},  {(13) \brack {3(13)}},  {(23) \brack {3(23)}},  {(13) \brack {(12)(23)}}, 
\end{equation}
where  $\displaystyle{{i \brack {i i}} }$ is non zero only for $k_i \geq 3$ ($ i = 1, 2, 3$).

Now we take into account the diffeomorphism
$$
G/\{e\}\cong (G\times \SO(k_1)\times\SO(k_2)\times\SO(k_3))/{\rm diag}(\SO(k_1)\times\SO(k_2)\times\SO(k_3))
$$ 
and consider left-invariant metrics on $G$ which are determined by the  
  $\Ad(\SO(k_1)\times\SO(k_2)\times\SO(k_3))$-invariant scalar products on $\fr{so}(k_1+k_2+k_3)$ given by
 \begin{equation} \label{metric001} 
\begin{array}{lll}
 \langle \  ,\   \rangle &=&  x_1 \, (-B) |_{\fr{so}(k_1)}+ x_2 \, (-B) |_{ \fr{so}(k_2)}+  x_3 \, (-B) |_{ \fr{so}(k_3)} 
 \\ &  & + x_{12} \,  (-B) |_{ \fr{m}_{12}}+ x_{13} \,  (-B) |_{ \fr{m}_{13}} + x_{23} \,  (-B) |_{ \fr{m}_{23}} 
 \end{array}.
\end{equation}

For $k_3=1$ we also consider  left-invariant metrics on $G$ which are determined by the  
$\Ad( \SO(k_1)\times\SO(k_2))$-invariant scalar products on $\fr{so}(k_1+k_2+1)$ of the form 
\begin{equation} \label{metric002} 
\begin{array}{l}
 \langle\  ,\   \rangle =  x_1 \, (-B) |_{\fr{so}(k_1)}+ x_2 \, (-B) |_{ \fr{so}(k_2)}+   
 x_{12} \,  (-B) |_{ \fr{m}_{12}}+ x_{13} \,  (-B) |_{ \fr{m}_{13}} + x_{23} \,  (-B) |_{ \fr{m}_{23}} .
\end{array}
\end{equation}

Finally, for $k_1=n-2$ and $k_2=k_3=1$ we consider
left-invariant metrics on $G$ which are determined by the  
$\Ad( \SO(n-2))$-invariant scalar products on $\fr{so}(n)$ of the form 
\begin{equation} \label{metric003} 
\begin{array}{l}
 \langle\  ,\   \rangle =  x_1 \, (-B) |_{\fr{so}(n-2)}+   
 x_{12} \,  (-B) |_{ \fr{m}_{12}}+ x_{13} \,  (-B) |_{ \fr{m}_{13}} + x_{23} \,  (-B) |_{ \fr{m}_{23}}. 
\end{array}
\end{equation}

For the scalar products (\ref{metric002}) the only non zero triplets are
\begin{equation*}
{1 \brack {11}},  {2 \brack {22}},    {(12) \brack {1(12)}},   {(13) \brack {1(13)}}, 
   {(12) \brack {2(12)}},    {(23) \brack {2(23)}},      {(13) \brack {(12)(23)}}, 
\end{equation*}
 and  for the  scalar products (\ref{metric003}) the only non zero triplets are
\begin{equation*}
{1 \brack {11}},   {(12) \brack {1(12)}},   {(13) \brack {1(13)}},    {(13) \brack {(12)(23)}}. 
\end{equation*}

\section{Naturally reductive metrics on the compact Lie groups $\SO(n)$} 

A Riemannian homogeneous space $(M=G/K, g)$ with reductive complement $\fr{m}$ of $\fr{k}$ in $\fr{g}$ is called {\it naturally reductive} if
$$
\langle [X, Y]_\fr{m}, Z\rangle +\langle Y, [X, Z]_\fr{m}\rangle=0 \quad\mbox{for all}\ 
X, Y, Z\in\fr{m}.
$$ 
Here $\langle\ , \  \rangle$ denotes the inner product on $\fr{m}$ induced from the Riemannian metric $g$.  Classical examples of naturally reductive homogeneous spaces include
irreducible symmetric spaces, isotropy irreducible homogeneous manifolds, 
and Lie groups with bi-invariant metrics.
In general it is not always easy to decide if a given homogeneous Riemannian manifold is naturally reductive, since one has to consider all possible transitive actions of subgroups
$G$ of the isometry group of $(M, g)$.

In \cite{DZ}  D'Atri and Ziller had investigated naturally reductive metrics among left-invariant metrics on compact Lie groups and gave a complete classification in the case of simple Lie groups.
Let $G$ be a compact, connected semisimple Lie group, $L$ a closed subgroup of $G$ and let
$\fr{g}$ be the Lie algebra of $G$ and $\fr{l}$ the subalgebra corresponding to $L$.
We denote by $Q$ the negative of the Killing form of $\fr{g}$.  Then $Q$ is an
$\Ad (G)$-invariant inner product on $\fr{g}$.

Let $\fr{m}$ be an orthogonal complement of $\fr{l}$ with respect to $Q$.  Then we have
$$
\fr{g}=\fr{l}\oplus\fr{m}, \quad \Ad(L)\fr{m}\subset\fr{m}.
$$
Let $\fr{l}=\fr{l}_0\oplus\fr{l}_1\oplus\cdots\oplus{l}_p$ be a decomposition of $\fr{l}$ into ideals, where $\fr{l}_0$ is the center of $\fr{l}$ and $\fr{l}_i$ $(i=1,\dots , p)$ are simple ideals of $\fr{l}$.
Let $A_0|_{\fr{l}_0}$ be an arbitrary metric on $\fr{l}_0$.

\begin{theorem}\label{DZ} (\cite[Theorem 1, p. 9]{DZ})  Under the notations above a left-invariant metric on $G$ of the form
\begin{equation}\label{natural}
\langle\  ,\  \rangle =x\cdot Q|_{\fr{m}}+A_0|_{\fr{l}_0}+u_1\cdot Q|_{\fr{l}_1}+\cdots
+ u_p\cdot Q|_{\fr{l}_p}, \quad (x, u_1, \dots , u_p >0)
\end{equation}
is naturally reductive with respect to $G\times L$, where $G\times L$ acts on $G$ by
$(g, l)y=gyl^{-1}$.

Moreover, if a left-invariant metric $\langle\ ,\ \rangle$ on a compact simple Lie group
$G$ is naturally reductive, then there is a closed subgroup $L$ of $G$ and the metric
$\langle\ ,\ \rangle$ is given by the form (\ref{natural}).
\end{theorem}

\medskip
For the Lie group  $\SO(n)$, we consider  left-invariant metrics determined by the $\Ad(\SO(k_1)\times\SO(k_2)\times\SO(k_3))$-invariant scalar products of the form (\ref{metric001}) where $n = k_1+k_2+k_3$. Recall that $K = \SO(k_1)\times\SO(k_2)\times\SO(k_3)$ with Lie algebra 
$\fr{k}=\fr{so}(k_1)\oplus\fr{so}(k_2)\oplus\fr{so}(k_3)$.

\begin{prop}\label{prop5.1}
If a left invariant metric $\langle\  ,\   \rangle$ of the form {\em (\ref{metric001})} on $\SO(n)$ is naturally reductive  with respect to $\SO(n)\times L$ for some closed subgroup $L$ of $\SO(n)$, 
then one of the following holds: 

{\em 1) } $x_1 = x_2 = x_{12}$,   $x_{13} = x_{23}$  \,   {\em  2)}   $x_2 = x_3 = x_{23}$,   $x_{12} = x_{13}$  \ 
{\em 3)}  $x_1 = x_3 = x_{13}$,   $x_{12} = x_{23}$,   \,   {\em  4)} $ x_{12} = x_{13} = x_{23} $.

Conversely, 
 if  one of the conditions {\em 1)},  {\em 2)},   {\em 3)},   {\em 4)} is satisfied, then the metric 
 $\langle\  , \   \rangle$ of the form {\em (\ref{metric001})}  is  naturally reductive  with respect to $\SO(n)\times L$ for some closed subgroup $L$ of $\SO(n)$.
  \end{prop}
  
\begin{proof}   Let ${\frak l}$ be the Lie algebra of  $L$. Then we have either ${\frak l} \subset {\frak k}$  or ${\frak l} \not\subset {\frak k}$. 
First we consider the case of  ${\frak l} \not\subset {\frak k}$. Let ${\frak h}$ be the subalgebra of ${\frak g}$ generated by ${\frak l}$ and ${\frak k}$. 
Since 
$ \fr{so}(k_1+k_2+k_3) = \fr{m}_1\oplus \fr{m}_2\oplus \fr{m}_3\oplus  \fr{m}_{12}\oplus  \fr{m}_{13}\oplus  \fr{m}_{23}$ is an irreducible decomposition as $\mbox{Ad}(K)$-modules, we see that the Lie algebra $\frak h$  contains  at least one of  ${\frak m}_{12}$,  ${\frak m}_{13}$, ${\frak m}_{23}$. 
We first consider the case that $\frak h$  contains ${\frak m}_{12}$. 
 Note that 
$\left[{\frak m}_{12}, {\frak m}_{12}\right] = {\frak m}_1 \oplus \fr{m}_2$ and $\fr{m}_1 \oplus {\frak m}_2  \oplus {\frak m}_{12}$ is a subalgebra $\fr{so}(k_1+k_2)$. 
Thus we see that $\frak h$ contains $\fr{so}(k_1+k_2) \oplus \fr{so}(k_3)$. 
If $\frak h =\fr{so}(k_1+k_2) \oplus \fr{so}(k_3)$, then we obtain an irreducible decomposition $ \fr{so}(k_1+k_2+k_3) = \frak h \oplus \fr{n}$, where $\fr{n} = {\frak m}_{13}\oplus{\frak m}_{23}$. Hence, the metric $\langle\   ,\   \rangle$ of the form (\ref{metric001}) satisfies $x_1 = x_2 = x_{12}$ and   $x_{13} = x_{23}$, so we obtain case 1).  Cases 2) and 3) are obtained by a  similar way. 

Now we consider the case ${\frak l} \subset {\frak k}$.  Since the  orthogonal complement
 ${\frak l}^{\bot}$ of ${\frak l}$ with respect to $-B$ contains the  orthogonal complement 
${\frak k}^{\bot}$ of ${\frak k}$, we see that ${\frak l}^{\bot} \supset {\frak m}_{12} \oplus  {\frak m}_{13}\oplus  {\frak m}_{23}$.   
Since the  invariant metric $\langle \  , \  \rangle$ is naturally reductive  with respect to $G\times L$,  
 it follows that  $ x_{12} = x_{13} = x_{23} $ by  Theorem \ref{DZ}.   
The converse is a direct consequence of Theorem \ref{DZ}.
\end{proof}

For the 
$\Ad( \SO(k_1)\times\SO(k_2))$-invariant metrics of the form 
(\ref{metric002}) the above proposition reduces to the following:

\begin{prop}\label{prop5.2}
If a left invariant metric $\langle\  ,  \  \rangle$ of the form {\em (\ref{metric002})} on $\SO(n)$ is naturally reductive  with respect to $\SO(n)\times L$ for some closed subgroup $L$ of $\SO(n)$, 
then one of the following holds: 

{\em 1) } $x_1 = x_2 = x_{12}$,   $x_{13} = x_{23}$,  \,   {\em  2)}   $x_2  = x_{23}$,   $x_{12} = x_{13}$,  \ 
{\em 3)}  $x_1  = x_{13}$,   $x_{12} = x_{23}$ ,  \,   {\em  4)} $ x_{12} = x_{13} = x_{23} $.

Conversely, 
 if  one of the conditions {\em 1)},  {\em 2)},   {\em 3)},   {\em 4)} is satisfied, then the metric $\langle \ , \   \rangle$ of the form {\em (\ref{metric002})}  is  naturally reductive  with respect to $\SO(n)\times L$ for some closed subgroup $L$ of $\SO(n)$.
  \end{prop}

  Finally, for the $\Ad(\SO(k_1))$-invariant metrics of the form (\ref{metric003}) we have the following:
  
  \begin{prop}\label{prop5.33}
If a left invariant metric $\langle\   ,\   \rangle$ of the form {\em (\ref{metric003})} on $\SO(n)$ is naturally reductive  with respect to $\SO(n)\times L$ for some closed subgroup $L$ of $\SO(n)$, 
then one of the following holds: 

{\em 1) } $x_1 =  x_{12}$,   $x_{13} = x_{23}$, \  \,     
{\em 2)}  $x_1  = x_{13}$,   $x_{12} = x_{23}$,  \  \,   {\em  3)} $ x_{12} = x_{13} $.

Conversely, 
 if  one of the conditions {\em 1)},  {\em 2)},   {\em 3)}   is satisfied, then the metric $\langle \  , \   \rangle$ of the form {\em (\ref{metric003})}  is  naturally reductive  with respect to $\SO(n)\times L$ for some closed subgroup $L$ of $\SO(n)$.
  \end{prop}

\section{The Ricci tensor for a class of left-invariant  metrics on $\SO(n)=\SO(k_1+k_2+k_3)$}

We will compute the Ricci tensor for the left-invariant metrics on $\SO(n)=\SO(k_1+k_2+k_3)$, determined by the
$\Ad( \SO(k_1)\times\SO(k_2)\times\SO(k_3))$-invariant scalar products of the form 
(\ref{metric001}). 
We use Lemma \ref{ric2} taking into account (\ref{triplets}), and we obtain the following:

\begin{prop}\label{lemma5.1}
The components  of  the Ricci tensor ${r}$ for the left-invariant metric $ \langle \  ,\   \rangle $ on $G$ defined by  {\em (\ref{metric001})}  are given as follows:  
\begin{equation}\label{eq13}
\small{\begin{array}{lll} 
r_1 &= & \displaystyle{\frac{1}{2 x_1} +
\frac{1}{4 d_1 } \biggl({1 \brack {11}}\frac{1}{x_{1}} +{1 \brack {(12)(12)}}  \frac{x_1}{{x_{12}}^2}+{1 \brack {(13)(13)}}  \frac{x_1}{{x_{13}}^2} \biggr)} \\ \\ & & 
 \displaystyle{- \frac{1}{2 d_1 } \biggl({1 \brack {11}}\frac{1}{x_{1}} +{(12) \brack {1(12)}}  \frac{1}{{x_{1}}}+{(13) \brack {1(13)}}  \frac{1}{{x_{1}}} \biggr),} 
 \\   \\
 r_2 &= & 
 \displaystyle{\frac{1}{2 x_2} +
\frac{1}{4 d_2 } \biggl({2 \brack {22}}\frac{1}{x_{2}} +{2 \brack {(12)(12)}}  \frac{x_2}{{x_{12}}^2}+{2 \brack {(23)(23)}}  \frac{x_2}{{x_{23}}^2} \biggr)} \\ \\ & & 
 \displaystyle{- \frac{1}{2 d_2 } \biggl({2 \brack {22}}\frac{1}{x_{2}} +{(12) \brack {2(12)}}  \frac{1}{{x_{2}}}+{(23) \brack {2(23)}}  \frac{1}{{x_{2}}} \biggr),} 
 \\  \\
 r_3 &= & 
 \displaystyle{\frac{1}{2 x_3} +
\frac{1}{4 d_3 } \biggl({3 \brack {33}}\frac{1}{x_{3}} +{3 \brack {(13)(13)}}  \frac{x_3}{{x_{13}}^2}+{3 \brack {(23)(23)}}  \frac{x_3}{{x_{23}}^2} \biggr)} \\ \\ & & 
 \displaystyle{- \frac{1}{2 d_3 } \biggl({3 \brack {33}}\frac{1}{x_{3}} +{(13) \brack {3(13)}}  \frac{1}{{x_{3}}}+{(23) \brack {3(23)}}  \frac{1}{{x_{3}}} \biggr),} 
\\  \\
r_{12} &= &  \displaystyle{\frac{1}{ 2 x_{12}} +\frac{1}{4 d_{12}}\biggl({(12) \brack {1 (12)}}  \frac{1}{x_{1}}\times 2 + {(12) \brack {2 (12)}}  \frac{1}{x_{2}} \times2+{(12) \brack {(13) (23)}}   \frac{x_{12}}{x_{13} x_{23} } \times 2 \biggr) } \\  \\ & & 
\displaystyle{-\frac{1}{ 2  d_{12}} \biggl( {1 \brack {(12) (12)}} \frac{x_1}{{x_{12}}^2} + {(12) \brack {(12) 1}} \frac{1}{x_{1}}+ {2 \brack {(12) (12)}}\frac{x_2}{{x_{12}}^2} +{(12) \brack {(12) 2}} \frac{1}{x_{2}} }  \\ \\ & & \displaystyle{+{(13) \brack {(12) (23)}}   \frac{x_{13}}{x_{12} x_{23} }+{(23) \brack {(12) (13)}}   \frac{x_{23}}{x_{12} x_{13} } \biggr) }, 
\\  \\  
r_{13}  &= &    \displaystyle{\frac{1}{ 2 x_{13}} +\frac{1}{4 d_{13}}\biggl({(13) \brack {1 (13)}}  \frac{1}{x_{1}}\times 2 + {(13) \brack {2 (13)}}  \frac{1}{x_{2}} \times2+{(13) \brack {(12) (23)}}   \frac{x_{13}}{x_{12} x_{23} } \times 2 \biggr) } \\  \\ & & 
\displaystyle{-\frac{1}{ 2  d_{13}} \biggl( {1 \brack {(13) (13)}} \frac{x_1}{{x_{13}}^2} + {(13) \brack {(13) 1}} \frac{1}{x_{1}}+ {3 \brack {(13) (13)}}\frac{x_3}{{x_{13}}^2} +{(13) \brack { (13) 3 }} \frac{1}{x_{3}} }  \\ \\ & & \displaystyle{+{(12) \brack {(13) (23)}}   \frac{x_{12}}{x_{13} x_{23} }+{(23) \brack {(13) (12)}}   \frac{x_{23}}{x_{13} x_{12} } \biggr) }, 
\end{array} }
\end{equation}
\begin{equation*}
\small{\begin{array}{lll} 
\\  \\
r_{23}  &= &    \displaystyle{\frac{1}{ 2 x_{23}} +\frac{1}{4 d_{23}}\biggl({(23) \brack {2 (23)}}  \frac{1}{x_{2}}\times 2 + {(23) \brack {3 (23)}}  \frac{1}{x_{3}} \times2+{(23) \brack {(12) (13)}}   \frac{x_{23}}{x_{12} x_{13} } \times 2 \biggr) } \\  \\ & & 
\displaystyle{-\frac{1}{ 2  d_{23}} \biggl( {2 \brack {(23) (23)}} \frac{x_2}{{x_{23}}^2} + {(23) \brack {(23) 2}} \frac{1}{x_{2}}+ {3 \brack {(23) (23)}}\frac{x_3}{{x_{23}}^2} +{(23) \brack {(23) 3}} \frac{1}{x_{3}} }  \\ \\ & & \displaystyle{+{(12) \brack {(23) (13)}}   \frac{x_{12}}{x_{23} x_{13} }+{(13) \brack {(23) (12)}}   \frac{x_{13}}{x_{23} x_{12} } \biggr) }, 
\end{array} }
\end{equation*}
where $n = k_1+k_2+k_3$. 
\end{prop}
%
%
%
%
We recall the following lemma from \cite{ADN1} (a  detailed proof was  given in \cite{ASS1}).

\begin{lemma}\label{lemma5.20} (\cite[Lemma 4.2]{ADN1})  For $a, b, c = 1, 2, 3$ and $(a - b)(b - c) (c - a) \neq 0$ the following relations hold:
\begin{equation}\label{eq14}
\begin{array}{lll} 
 \displaystyle{{a \brack {a a}} = \frac{k_a (k_a -1)(k_a -2)}{2 (n -2)} },   &  \displaystyle{{a \brack {(a b) (a b)}} = \frac{k_a  k_b (k_a -1)}{2 (n -2)} }, &  \displaystyle{{(a c) \brack {(a b ) (b c)}} = \frac{k_a  k_b  k_c}{2 (n -2)} }. 
\end{array} 
\end{equation}
\end{lemma} 
 
 By using the above lemma, we can now obtain the components of the Ricci tensor for the metrics we are considering in this work.

 \begin{prop}\label{prop5.3}
The components  of  the Ricci tensor ${r}$ for the left-invariant metric $ \langle \   ,\   \rangle $ on $G$ defined by  {\em(\ref{metric001})} are given as follows:  
\begin{equation}\label{eq17}
\left. 
\small{\begin{array}{l} 
r_1 =  \displaystyle{\frac{k_1-2}{4 (n -2)  x_1} +
\frac{1}{4 (n -2) } \biggl(k_2 \frac{x_1}{{x_{12}}^2}} +k_3 \frac{x_1}{{x_{13}}^2} \biggr), 
 \\   \\
 r_2 = 
\displaystyle{\frac{k_2-2}{4 (n -2)  x_2} +
\frac{1}{4 (n -2)} \biggl(k_1 \frac{x_2}{{x_{12}}^2} +k_3 \frac{x_2}{{x_{23}}^2} \biggr),} 
\\  \\
 r_3 =  
\displaystyle{\frac{k_3-2}{4 (n -2)  x_3} +
\frac{1}{4 (n -2)} \biggl(k_1 \frac{x_3}{{x_{13}}^2} +k_2 \frac{x_3}{{x_{23}}^2} \biggr),} 
\\  \\
r_{12} =   \displaystyle{\frac{1}{ 2 x_{12}} +\frac{k_3}{4 (n -2)}\biggl(\frac{x_{12}}{x_{13} x_{23}} - \frac{x_{13}}{x_{12} x_{23}} - \frac{x_{23}}{x_{12} x_{13}}\biggr) }
\displaystyle{-
\frac{1}{4 (n -2)} \biggl( \frac{(k_1-1) x_1}{{x_{12}}^2} +  \frac{(k_2-1) x_2}{{x_{12}}^2} \biggr)},
\\  \\
r_{13}  =   \displaystyle{\frac{1}{  2 x_{13}} +\frac{k_2}{4 (n -2)}\biggl(\frac{x_{13}}{x_{12} x_{23}} - \frac{x_{12}}{x_{13} x_{23}} - \frac{x_{23}}{x_{12} x_{13}}\biggr)} 
\displaystyle{-
\frac{1}{4 (n -2)} \biggl( \frac{(k_1-1) x_1}{{x_{13}}^2} + \frac{(k_3-1) x_3}{{x_{13}}^2}  \biggr)}, 
\\  \\
r_{23}  =   \displaystyle{\frac{1}{ 2 x_{23}} +\frac{k_1}{4 (n -2)}\biggl(\frac{x_{23}}{x_{13} x_{12}} - \frac{x_{13}}{x_{12} x_{23}} - \frac{x_{12}}{x_{23} x_{13}}\biggr)}
\displaystyle{-
\frac{1}{4 (n -2)} \biggl( \frac{(k_2-1) x_2}{{x_{23}}^2} + \frac{(k_3-1) x_3}{{x_{23}}^2} \biggr)}, 
\end{array} } \right\}
\end{equation}
where $n = k_1+k_2+k_3$. 
\end{prop}

For  the  $\Ad( \SO(k_1)\times\SO(k_2))$-invariant metrics on $G$ of the form 
 (\ref{metric002}), 
   Proposition \ref{prop5.3} reduces to

\begin{prop}\label{prop5.4}
The components  of  the Ricci tensor ${r}$ for the left-invariant metric $ \langle \   ,\    \rangle $ on $G$ defined by  {\em(\ref{metric002})}  are given as follows:  
\begin{equation}\label{eq17}
\left. 
\small{\begin{array}{l} 
r_1 = \displaystyle{\frac{k_1-2}{4 (n -2)  x_1} +
\frac{1}{4 (n -2) } \biggl(k_2 \frac{x_1}{{x_{12}}^2}} + \frac{x_1}{{x_{13}}^2} \biggr), 
 \\   \\
 r_2 = 
\displaystyle{\frac{k_2-2}{4 (n -2)  x_2} +
\frac{1}{4 (n -2)} \biggl(k_1 \frac{x_2}{{x_{12}}^2} + \frac{x_2}{{x_{23}}^2} \biggr),} 
\\  \\
r_{12} =  \displaystyle{\frac{1}{ 2 x_{12}} +\frac{1}{4 (n -2)}\biggl(\frac{x_{12}}{x_{13} x_{23}} - \frac{x_{13}}{x_{12} x_{23}} - \frac{x_{23}}{x_{12} x_{13}} -
  \frac{(k_1-1) x_1}{{x_{12}}^2} - \frac{ (k_2-1) x_2}{{x_{12}}^2} \biggr)},
\\  \\
r_{13}  =  \displaystyle{\frac{1}{  2 x_{13}} +\frac{1}{4 (n -2)}\biggl({k_2} \left(\frac{x_{13}}{x_{12} x_{23}} - \frac{x_{12}}{x_{13} x_{23}} - \frac{x_{23}}{x_{12} x_{13}}  \right)  
-   \frac{ (k_1-1) x_1}{{x_{13}}^2} \biggr)}, 
\\  \\
r_{23}  =  \displaystyle{\frac{1}{ 2 x_{23}} +\frac{1}{4 (n -2)}\biggl(k_1\left(\frac{x_{23}}{x_{13} x_{12}} - \frac{x_{13}}{x_{12} x_{23}} - \frac{x_{12}}{x_{23} x_{13}}\right)   
 -    \frac{ (k_2-1) x_2}{{x_{23}}^2} \biggr)}, 
\end{array} } \right\}
\end{equation}
where $n = k_1+k_2+1$. 
\end{prop}

Finally, for  the  $\Ad( \SO(k_1))$-invariant metrics on $G$ of the form 
 (\ref{metric003}), 
   Proposition \ref{prop5.3} reduces to

\begin{prop}\label{prop5.5}
The components  of  the Ricci tensor ${r}$ for the left-invariant metric $ \langle \  \ ,\ \ \rangle $ on $G$ defined by  {\em(\ref{metric003})}  are given as follows:  
\begin{equation}\label{eq17}
\left. 
\small{\begin{array}{l} 
r_1 = \displaystyle{\frac{n-4}{4 (n -2)  x_1} +
\frac{1}{4 (n -2) } \biggl( \frac{x_1}{{x_{12}}^2}} + \frac{x_1}{{x_{13}}^2} \biggr), 
 \\   \\
r_{12} =  \displaystyle{\frac{1}{ 2 x_{12}} +\frac{1}{4 (n -2)}\biggl(\frac{x_{12}}{x_{13} x_{23}} - \frac{x_{13}}{x_{12} x_{23}} - \frac{x_{23}}{x_{12} x_{13}} -
  \frac{(n-3) x_1}{{x_{12}}^2}  \biggr)},
\\  \\
r_{13}  =  \displaystyle{\frac{1}{  2 x_{13}} 
  +\frac{1}{4 (n -2)}\biggl( \frac{x_{13}}{x_{12} x_{23}} - \frac{x_{12}}{x_{13} x_{23}} - \frac{x_{23}}{x_{12} x_{13}}   
-   \frac{ (n-3) x_1}{{x_{13}}^2} \biggr)}, 
\\  \\
r_{23}  =  \displaystyle{\frac{1}{ 2 x_{23}} 
+\frac{1}{4 }\left(\frac{x_{23}}{x_{13} x_{12}} - \frac{x_{13}}{x_{12} x_{23}} - \frac{x_{12}}{x_{23} x_{13}}\right)   
}, 
\end{array} } \right\}
\end{equation}
where $n = k_1+2$. 
\end{prop}

\section{Left-invariant Einstein metrics on $\SO(n)$} 

In the present section we provide detailed proofs on how to obtain left-invariant Einstein metrics which are not naturally reductive for the Lie groups $\SO(7)$, $\SO(8)$ and $\SO(n)$, $n\ge 9$. 
For $\SO(7)$ and  $\SO(8)$  we also describe left-invariant Einstein metrics which are naturally reductive.  The other compact Lie groups $\SO(n)=\SO(k_1+k_2+k_3)$  for $n\ge 7$ and for all possible  
$\Ad( \SO(k_1)\times\SO(k_2)\times\SO(k_3))$-invariant  metrics of the forms (\ref{metric001}) and (\ref{metric002})
can be treated in an analogous manner and we omit the proofs.  We summarise all the results at the end of Section 7.  Here, we provide information about solving or proving existence of solutions of algebraic systems of equations.  
These solutions correspond to Einstein metrics which are not naturally reductive. 

\begin{prop}\label{prop1}
The Lie group  $\SO(7)$ admits at least one left-invariant Einstein metric determined by the
 $\Ad( \SO(3)\times\SO(3))$-invariant scalar product of the form {\rm (\ref{metric002})}, which is not naturally reductive. 
\end{prop}
 
\begin{proof} 
 This is the case when  $k_1= k_2 =3$ and $k_3=1$. 
  From Proposition \ref{prop5.4}, we see that  the components  of  the Ricci tensor ${r}$ for the invariant metric are given by 
 \begin{equation}\label{eq331}
\left. 
\small{\begin{array}{lll} 
r_1 &= & \displaystyle{\frac{1}{20 x_1} +
\frac{1}{20} \biggl( 3 \frac{x_1}{{x_{12}}^2}} + \frac{x_1}{{x_{13}}^2} \biggr),  \quad 
 r_2   \, \, \, =  \, \, \, 
\displaystyle{\frac{1}{20 x_2} +
\frac{1}{20} \biggl( 3 \frac{x_2}{{x_{12}}^2} +  \frac{x_2}{{x_{23}}^2} \biggr),} 
\\  \\
r_{12} &= &  \displaystyle{\frac{1}{ 2 x_{12}} +\frac{1}{20}\biggl(\frac{x_{12}}{x_{13} x_{23}} - \frac{x_{13}}{x_{12} x_{23}} - \frac{x_{23}}{x_{12} x_{13}}\biggr) } 
\displaystyle{-
\frac{1}{10}  \biggl(  \frac{x_1}{{x_{12}}^2}+  \frac{x_2}{{x_{12}}^2} \biggr)},
\\  \\
r_{23}  &= &  \displaystyle{\frac{1}{ 2 x_{23}} +\frac{3}{20}\biggl(\frac{x_{23}}{x_{13} x_{12}} - \frac{x_{13}}{x_{12} x_{23}} - \frac{x_{12}}{x_{23} x_{13}}\biggr) -
\frac{1}{10}   \frac{x_2}{{x_{23}}^2}},
\\  \\
r_{13}  &= &  \displaystyle{\frac{1}{  2 x_{13}} +\frac{3}{20}\biggl(\frac{x_{13}}{x_{12} x_{23}} - \frac{x_{12}}{x_{13} x_{23}} - \frac{x_{23}}{x_{12} x_{13}}\biggr)  -
\frac{1}{10}   \frac{x_1}{{x_{13}}^2}}.  
\end{array} } \right\}
\end{equation}

 We consider the system of equations   
 \begin{equation}\label{eq331a} 
 r_1 = r_2, \, \, r_2 =  r_{12}, \, \, r_{12} = r_{23}, \,\,  r_{23} = r_{13}. 
 \end{equation}
Then finding Einstein metrics of the form (\ref{metric002})  reduces  to finding the positive solutions of system (\ref{eq331a}),  and  we normalize  our equations by putting $x_{23}=1$. Then we obtain the system of equations: 
 \begin{equation}\label{eq331_2a} 
\left. { \begin{array}{l}
g_1={x_{1}}^2 {x_{12}}^2 {x_{2}}+3 {x_{1}}^2 {x_{13}}^2 {x_{2}}-{x_{1}} {x_{12}}^2 {x_{13}}^2 {x_{2}}^2-{x_{1}} {x_{12}}^2 {x_{13}}^2 \\
 \quad  \quad  -3 {x_{1}}{x_{13}}^2 {x_{2}}^2+{x_{12}}^2 {x_{13}}^2 {x_{2}} = 0,  \\ 
g_2=  2{x_{1}} {x_{13}} {x_{2}} - {x_{12}}^3 {x_{2}} +{x_{12}}^2 {x_{13}} {x_{2}}^2+{x_{12}}^2
   {x_{13}}+{x_{12}} {x_{13}}^2 {x_{2}} \\
    \quad  \quad -10 {x_{12}} {x_{13}} {x_{2}}+{x_{12}} {x_{2}}+5 {x_{13}}{x_{2}}^2= 0, \\ 
g_3= -{x_{1}} {x_{13}}+2 {x_{12}}^3+{x_{12}}^2 {x_{13}} {x_{2}}-5 {x_{12}}^2 {x_{13}}+{x_{12}}
   {x_{13}}^2 \\
    \quad  \quad +5 {x_{12}} {x_{13}}-2 {x_{12}}-{x_{13}}{x_{2}} = 0, \\
 g_4= {x_{1}} {x_{12}}-{x_{12}} {x_{13}}^2 {x_{2}}+5 {x_{12}} {x_{13}}^2-5 {x_{12}} {x_{13}}-3
   {x_{13}}^3+3 {x_{13}}=0. 
\end{array} } \right\} 
\end{equation}

   We consider a polynomial ring $R= {\mathbb Q}[z, x_1, x_2, x_{12}, x_{13}] $ and an ideal $I$ generated by 
$\{ g_1, \, g_2, $  $ g_3, g_4,  \,z \,x_1 \, x_2 \, x_{12} \, x_{13} -1\}  
$  to find non zero solutions of equations (\ref{eq331_2a}). 
We take a lexicographic order $>$  with $ z > x_1 >  x_2 > x_{12} > x_{13}$ for a monomial ordering on $R$. Then, by the aid of computer, we see that a  Gr\"obner basis for the ideal $I$ contains the  polynomial
$$({x_{13}}-1) \left(6 {x_{13}}^3-44 {x_{13}}^2+90
   {x_{13}}-45\right) \left(45 {x_{13}}^3-90 {x_{13}}^2+44
   {x_{13}}-6\right) \, h_{1}(x_{13}),$$
where $h_{1}(x_{13})$ is a polynomial of   $x_{13}$  given by 
\begin{eqnarray*}  & & 
h_{1}(x_{13}) = 9078544800000
   {x_{13}}^{24}-87978150000000 {x_{13}}^{23}+416122213455000
   {x_{13}}^{22} \\  & & -1222223075437500
   {x_{13}}^{21}+2532878590309970
   {x_{13}}^{20}-4171390831990050
   {x_{13}}^{19} \\  & & +5900094406718764
   {x_{13}}^{18}-7070644584919459
   {x_{13}}^{17}+6230617318198202
   {x_{13}}^{16} \\  & & -4091340309226802
   {x_{13}}^{15}+1722695469975774
   {x_{13}}^{14}+983550542994755
   {x_{13}}^{13} \\  & & -2624020500593532
   {x_{13}}^{12}+983550542994755
   {x_{13}}^{11}+1722695469975774
   {x_{13}}^{10} \\  & & -4091340309226802 {x_{13}}^9+6230617318198202
   {x_{13}}^8-7070644584919459 {x_{13}}^7 \\  & & +5900094406718764
   {x_{13}}^6-4171390831990050 {x_{13}}^5+2532878590309970
   {x_{13}}^4 \\  & & -1222223075437500 {x_{13}}^3+416122213455000
   {x_{13}}^2-87978150000000 {x_{13}}
   \\  & & +9078544800000. 
\end{eqnarray*}
By solving the equation $ h_{1}(x_{13})=0$ numerically, we obtain {\it two} positive solutions $x_{13}= a_{13}$ and $x_{13}= b_{13}$ which are given approximately as
 $ a_{13} \approx 0.4254295, \, \,  b_{13} \approx 2.350565. $
 We also see that the Gr\"obner basis for the ideal $I$ contains the polynomials 
 $$x_{12} - w_{12}(x_{13}), \quad x_1 -w_1 (x_{13}), \quad x_2 -w_2 (x_{13}), $$
 where $ w_{12}(x_{13})$, $ w_{1}(x_{13})$ and $ w_{2}(x_{13})$ are polynomials of $x_{13}$  with rational coefficients. By substituting the values $ a_{13}$ and  $b_{13}$ for $x_{13}$  into $w_{12}(x_{13})$, $ w_{1}(x_{13})$ and $w_{2}(x_{13})$, we obtain two positive solutions of the system of equations $\{ g_1=0, g_2=0, g_3=0,  g_4=0  \}$ approximately as
 $$( x_{13}, x_{12}, x_1,  x_2)  \approx ( 0.4254295, \,0.9312204, \, 0.1200109, \, 0.1122291 ), $$  
 $$( x_{13}, x_{12},  x_1,  x_2)  \approx ( 2.350565, \, 2.188895, \, 0.2638018, \, 0.2820935). 
 $$
We substitute these  values into the system   (\ref{eq331}) together with $x_{23} =1$. Then we obtain that 
$r_1 = r_2 =r_{12} = r_{23} = r_{13} \approx  0. 470542$ and  $r_1 = r_2 =r_{12} = r_{23} = r_{13} \approx  0.200182$ respectively. We multiply these  solutions by a scale factor and we obtain the two solutions 
 $$( x_1,  x_2, x_{12}, x_{23}, x_{13})  \approx ( 0.0564701, \, 0.0528085, \, 0.438178, \, 0.470542, \, 0.20018 ), $$  
 $$(  x_1,  x_2, x_{12}, x_{23}, x_{13})  \approx ( 0.0528085, \, 0.0564701, \, 0.438178, \, 0.20018, \, 0.470542). 
 $$
for the system of equations  
$$
r_1 = r_2 =r_{12} = r_{23} = r_{13}  =1,
$$ 
so we see that {\it these two solutions are isometric}.  Note that this metric is not naturally reductive from  Proposition \ref{prop5.2}. 

Now we consider the case $$({x_{13}}-1) \left(6 {x_{13}}^3-44 {x_{13}}^2+90
   {x_{13}}-45\right) \left(45 {x_{13}}^3-90 {x_{13}}^2+44
   {x_{13}}-6\right) =0. $$
   
  We   consider the ideals $J_1$ generated by 
$\{ g_1, \, g_2, $  $ g_3, g_4,  \,z \,x_1 \, x_2 \, x_{12} \, x_{13} -1, 6 {x_{13}}^3-44 {x_{13}}^2+90
   {x_{13}}-45 \}  
$,  $J_2$ generated by 
$\{ g_1, \, g_2, $  $ g_3, g_4,  \,z \,x_1 \, x_2 \, x_{12} \, x_{13} -1,   45 {x_{13}}^3-90 {x_{13}}^2+44
   {x_{13}}-6  \}$ and  $J_3$
generated by 
$\{ g_1, \, g_2, $  $ g_3, g_4,  \,z \,x_1 \, x_2 \, x_{12} \, x_{13} -1, {x_{13}}-1 \}$ of the  polynomial ring $R= {\mathbb Q}[z, x_1, x_2, x_{12}, x_{13}] $. 

We take a lexicographic order $>$  with $ z > x_1 >  x_2 > x_{12} > x_{13}$ for a monomial ordering on  $R= {\mathbb Q}[z, x_1, x_2, x_{12}, x_{13}] $. Then, by the aid of computer, we see that a  Gr\"obner basis for the ideal $J_1$ is given by 
\begin{eqnarray*}  & &  \{6 {x_{13}}^3-44 {x_{13}}^2+90 {x_{13}}-45, x_{12}-{x_{13}},x_{2}-1, x_{1}+{x_{13}}^2-5 {x_{13}}+3, \\
& & -804 {x_{13}}^2+5284 {x_{13}}+405 z-8112 \}.   
   \end{eqnarray*}   
  By solving the equation $6 {x_{13}}^3-44 {x_{13}}^2+90 {x_{13}}-45=0$ numerically, we obtain three positive solutions of the system of equations $\{ g_1=0, g_2=0, g_3=0,  g_4=0, x_{13} -1, 6 {x_{13}}^3-44 {x_{13}}^2+90 {x_{13}}-45=0  \}$ approximately as
\begin{eqnarray*}  & & ( x_{13}, x_{12}, x_1,  x_2, x_{23} )  \approx ( 4.16278, \,4.16278, \, 0.485171, \,  1, \,  1 ), \\   
& & ( x_{13}, x_{12}, x_1,  x_2, x_{23} )  \approx ( 2.42874, \,2.42874, \, 3.24492, \,  1, \,  1 ), \\  
& & ( x_{13}, x_{12},  x_1,  x_2, x_{23} )  \approx ( 0.741818, \, 0.741818, \, 0.158797, \, 1, \,  1 ).     \end{eqnarray*}   
 We substitute these  values into the system   (\ref{eq331}). Then we obtain that 
$r_1 = r_2 =r_{12} = r_{23} = r_{13} \approx  0.108656$, $r_1 = r_2 =r_{12} = r_{23} = r_{13} \approx  0.125429$ and  $r_1 = r_2 =r_{12} = r_{23} = r_{13} \approx 0.372581$ respectively. We multiply these  solutions by a scale factor and we obtain three solutions 
\begin{eqnarray*}  & & ( x_{13}, x_{12}, x_1,  x_2, x_{23} )  \approx ( 0.452311, \,0.452311, \, 0.0527168, \,  0.108656, \,  0.108656 ), \\   
& & ( x_{13}, x_{12}, x_1,  x_2, x_{23} )  \approx ( 0.304634, \, 0.304634, \, 0.407007, \,  0.125429, \,  0.125429 ), \\   
& & ( x_{13}, x_{12},  x_1,  x_2, x_{23} )  \approx ( 0.276388, \, 0.276388, \, 0.0591647, \, 0.372581, \,  0.372581 ). 
\end{eqnarray*}  
for the system of equations  
$$
r_1 = r_2 =r_{12} = r_{23} = r_{13}  =1,
$$

 We  also see that a  Gr\"obner basis for the ideal $J_2$ is given by 
 \begin{eqnarray*}  & &  \{ 45 {x_{13}}^3-90 {x_{13}}^2+44 {x_{13}}-6,  x_{12}-1,45 {x_{13}}^2-72 {x_{13}}+6 {x_{2}}+14, {x_{1}}-{x_{13}}, \\
& & -4545 {x_{13}}^2+8010 {x_{13}}+6 z-2554 \}.   
   \end{eqnarray*}   
 By solving the equation $45 {x_{13}}^3-90 {x_{13}}^2+44 {x_{13}}-6=0$ numerically, we obtain three positive solutions of the system of equations $\{ g_1=0, g_2=0, g_3=0,  g_4=0, x_{13} -1, 45 {x_{13}}^3-90 {x_{13}}^2+44 {x_{13}}-6=0  \}$ approximately as
\begin{eqnarray*}  & & ( x_{13}, x_{12}, x_1,  x_2, x_{23} )  \approx ( 0.240224, \,1, \, 0.240224, \,  0.11655, \,  1 ), \\   
 & & ( x_{13}, x_{12}, x_1,  x_2, x_{23} )  \approx ( 0.411737, \, 1, \, 0.411737, \,  1.33605, \,  1 ), \\  
 & & ( x_{13}, x_{12},  x_1,  x_2, x_{23} )  \approx ( 1.34804, \,1, \,  1.34804,  \, 0.214064, \,  1 ).  
 \end{eqnarray*}  

We substitute these  values into the system   (\ref{eq331}). Then we obtain that 
$r_1 = r_2 =r_{12} = r_{23} = r_{13} \approx  0.452311$,  $r_1 = r_2 =r_{12} = r_{23} = r_{13} \approx  0.304634$ and  $r_1 = r_2 =r_{12} = r_{23} = r_{13} \approx  0.276388$ respectively. We multiply these  solutions by a scale factor and we obtain three solutions 
\begin{eqnarray*} & & ( x_1,  x_2, x_{12}, x_{23}, x_{13})  \approx ( 0.108656, \, 0.0527168, \, 0.452311, \, 0.452311, \, 0.108656 ), \\ 
& & ( x_1,  x_2, x_{12}, x_{23}, x_{13})  \approx ( 0.125429, \, 0.407007, \, 0.304634, \, 0.304634, \, 0.125429 ),\\   
 & & (  x_1,  x_2, x_{12}, x_{23}, x_{13})  \approx ( 0.372581, \, 0.0591647, \, 0.276388, \, 0.276388, \, 0.372581 ). 
 \end{eqnarray*} 
for the system of equations  
$$
r_1 = r_2 =r_{12} = r_{23} = r_{13}  =1.
$$ 
Note that, from these two cases,  we obtain {\it three Einstein metrics} up to isometry. 

Now we consider the case of  ideal $J_3$. We see that a  Gr\"obner basis for the ideal $J_3$ is given by 
 \begin{eqnarray*}  & & \{{x_ {13}} - 1, ({x_ {12}} - 1) (3 {x_ {12}} - 2) \left(2 {x_ {12}}^4 - 5 {x_ {12}}^3 + 19 {x_ {12}}^2 - 35 {x_ {12}} + 26 \right), \\
& & -24 {x_ {12}}^5 + 70 {x_ {12}}^4 - 267{x_ {12}}^3 + 550 {x_ {12}}^2 - 522 {x_ {12}} + 63 {x_ {2}} + 130, \\
 & & 63 {x_ {1}} - 24{x_ {12}}^5 + 70 {x_ {12}}^4 - 267{x_ {12}}^3 + 550{x_ {12}}^2 - 522
   {x_ {12}} + 130, \\
 & &  -569934 {x_ {12}}^5 + 2980005{x_ {12}}^4 - 8034374 {x_ {12}}^3 + 22826670 {x_ {12}}^2  \\ 
 & & - 25019231 {x_ {12}} + 246064 z + 7570800 \}.   
   \end{eqnarray*}  
 By solving the equation $2 {x_ {12}}^4 - 5 {x_ {12}}^3 + 19 {x_ {12}}^2 - 35 {x_ {12}} + 26 =0$ numerically, we see that there are no real solutions.   For $x_ {12} =1$, we obtain that $x_1 = x_2= x_{13} = x_{23} =1$, that is, the metric is bi-invariant. For $x_ {12} = 2/3$, we obtain that $x_1 = x_2= 2/3$ and $x_{13}= x_{23} =1$, hence we obtain {\it two Einstein metrics} up to isometry. 
 
From Proposition \ref{prop5.2} it follows that the five  above metrics  obtained are 
all naturally reductive  with respect to $\SO(7)\times L$, where   $L$ is
a closed subgroup of $ \SO(7)$, which is  either $\SO(3)\times\SO(4)$,
$\SO(6)$ or $\SO(7)$. 
\end{proof}
   
\smallskip
\begin{prop}\label{prop2}
The Lie group  $\SO(8)$ admits at least two  (non isometric) 
left-invariant Einstein metrics determined by the
$\Ad( \SO(4)\times\SO(3))$-invariant scalar products of the form {\rm (\ref{metric002})}, which are not naturally reductive. 
 \end{prop}
\begin{proof}
 This is the case when $k_1= 4, k_2 =3$ and $k_3=1$. 
  From Proposition \ref{prop5.4}, we see that  the components  of  the Ricci tensor ${r}$ for the invariant metric are given by 
 \begin{equation}\label{eq431aa}
\left. 
\small{\begin{array}{lll} 
r_1 &= & \displaystyle{\frac{1}{12 x_1} +
\frac{1}{24} \biggl( 3 \frac{x_1}{{x_{12}}^2}} + \frac{x_1}{{x_{13}}^2} \biggr),  \quad 
 r_2   \, \, \, =  \, \, \, 
\displaystyle{\frac{1}{24 x_2} +
\frac{1}{24} \biggl( 4 \frac{x_2}{{x_{12}}^2} +  \frac{x_2}{{x_{23}}^2} \biggr),} 
\\  \\
r_{12} &= &  \displaystyle{\frac{1}{ 2 x_{12}} +\frac{1}{24}\biggl(\frac{x_{12}}{x_{13} x_{23}} - \frac{x_{13}}{x_{12} x_{23}} - \frac{x_{23}}{x_{12} x_{13}}\biggr) } 
\displaystyle{-\frac{1}{24}  \biggl( 3\frac{x_1}{{x_{12}}^2}+  2 \frac{x_2}{{x_{12}}^2} \biggr)},
\\  \\
r_{23}  &= &  \displaystyle{\frac{1}{ 2 x_{23}} +\frac{1}{6}\biggl(\frac{x_{23}}{x_{13} x_{12}} - \frac{x_{13}}{x_{12} x_{23}} - \frac{x_{12}}{x_{23} x_{13}}\biggr) -
\frac{1}{12}   \frac{x_2}{{x_{23}}^2}},
\\  \\
r_{13}  &= &  \displaystyle{\frac{1}{  2 x_{13}} +\frac{1}{8}\biggl(\frac{x_{13}}{x_{12} x_{23}} - \frac{x_{12}}{x_{13} x_{23}} - \frac{x_{23}}{x_{12} x_{13}}\biggr)  -
\frac{1}{8}   \frac{x_1}{{x_{13}}^2}}.  
  
\end{array} } \right\}
\end{equation}

 We consider the system of equations   
 \begin{equation}\label{eq431} 
 r_1 = r_2, \, \, r_2 =  r_{12}, \, \, r_{12} = r_{23}, \,\,  r_{23} = r_{13}. 
 \end{equation}
Then finding Einstein metrics of the form (\ref{metric002})  reduces  to finding the positive solutions of system (\ref{eq431}),  and  we normalize  our equations by putting $x_{23}=1$. Then we have the system of equations: 
 \begin{equation}\label{eq431_2a} 
\left. { \begin{array}{l}
g_1= {x_{1}}^2 {x_{12}}^2 {x_{2}}+3 {x_{1}}^2 {x_{13}}^2 {x_{2}}-{x_{1}} {x_{12}}^2 {x_{13}}^2
   {x_{2}}^2-{x_{1}} {x_{12}}^2 {x_{13}}^2 \\
 \quad  \quad  -4 {x_{1}} {x_{13}}^2 {x_{2}}^2+2 {x_{12}}^2 {x_{13}}^2 {x_{2}} = 0,  \\ 
g_2=  3 {x_{1}} {x_{13}} {x_{2}}-{x_{12}}^3 {x_{2}}+{x_{12}}^2 {x_{13}} {x_{2}}^2
+{x_{12}}^2 {x_{13}}+{x_{12}} {x_{13}}^2 {x_{2}} \\
    \quad  \quad -12 {x_{12}} {x_{13}} {x_{2}}+{x_{12}} {x_{2}}+6 {x_{13}} {x_{2}}^2= 0, \\ 
g_3= -3 {x_{1}} {x_{13}}+5{x_{12}}^3+2 {x_{12}}^2 {x_{13}} {x_{2}}-12 {x_{12}}^2 {x_{13}}+3 {x_{12}} {x_{13}}^2\\
    \quad  \quad +12 {x_{12}} {x_{13}}-5 {x_{12}}-2 {x_{13}}{x_{2}}  = 0, \\
 g_4= 3 {x_{1}} {x_{12}}-{x_{12}}^2{x_{13}}-2 {x_{12}} {x_{13}}^2 {x_{2}}+12 {x_{12}} {x_{13}}^2
 -12 {x_{12}} {x_{13}} \\
\quad  \quad  -7 {x_{13}}^3+7 {x_{13}}=0. 
\end{array} } \right\} 
\end{equation}

  We consider a polynomial ring $R= {\mathbb Q}[z, x_1, x_2, x_{12}, x_{13}] $ and an ideal $I$ generated by 
$\{ g_1, \, g_2, $  $ g_3, g_4,  \,z \,x_1 \, x_2 \, x_{12} \, x_{13} -1\}  
$  to find non zero solutions of equations (\ref{eq431_2a}). 
We take a lexicographic order $>$  with $ z > x_1 >  x_2 > x_{12} > x_{13}$ for a monomial ordering on $R$. Then, by the aid of computer, we see that a  Gr\"obner basis for the ideal $I$ contains the  polynomial
$$({x_{13}}-5) ({x_{13}}-1) \left(7
   {x_{13}}^2-24 {x_{13}}+14\right)
   \left(287 {x_{13}}^3-625 {x_{13}}^2+369
   {x_{13}}-63\right) \, h_2(x_{13}),$$
where $h_{2}(x_{13})$ is a polynomial of   $x_{13}$  given by 
\begin{eqnarray*}  & & 
h_{2}(x_{13}) =  5426775507148489670400 
   {x_{13}}^{28}-85161185092622977873920 
   {x_{13}}^{27} \\ & &
   +643415930216926223949312 
   {x_{13}}^{26}-3054548385819855899001216 
   {x_{13}}^{25}\\ & &
   +10179140499777121100664800 
   {x_{13}}^{24}-25585147362416655835236384 
   {x_{13}}^{23}\\ & &
   +51380426324079059150364272 
   {x_{13}}^{22}-85934185504663087173249048 
   {x_{13}}^{21}\\ & &
   +120352447918421302289568863 
   {x_{13}}^{20}-136938372384910964649260802 
   {x_{13}}^{19}\\ & &
   +121268417379459335461167457 
   {x_{13}}^{18}-78483773118912467818333590 
   {x_{13}}^{17}\\ & &
   +32048679980888195807658286 
   {x_{13}}^{16}-21037081214018592447662850 
   {x_{13}}^{15}\\ & &
   +96567724403906545251348604 
   {x_{13}}^{14}-279673822213789859470643520 
   {x_{13}}^{13}\\ & &
   +527833035046902978479331387 
   {x_{13}}^{12}-769632045866390647274523642 
   {x_{13}}^{11}\\ & &
   +937521733316934021780397473 
   {x_{13}}^{10}-973318915329328329165562374 
   {x_{13}}^9\\ & &
   +864907599634224063462448416 
   {x_{13}}^8-664413545084655303518836950 
   {x_{13}}^7\\ & &
   +442175543674339070418041970 
   {x_{13}}^6-249282932584983174857359764 
   {x_{13}}^5\\ & &+114233981412525395978707920 
   {x_{13}}^4-40474281023127469650239100 
   {x_{13}}^3\\ & &+10382320721058779134026000 
   {x_{13}}^2-1735984447231701886065000 
   {x_{13}}\\ & &
   +146138820428187141975000. 
\end{eqnarray*}
 By solving the equation $ h_2(x_{13})=0$ numerically, we obtain {\it two} positive solutions $x_{13}= a_{13}$ and $x_{13}= b_{13}$ which are given approximately as
 $ a_{13} \approx  0.48183112, \, \,  b_{13} \approx  2.7966957. $
 We also see that the Gr\"obner basis for the ideal $I$ contains the polynomials 
 $$x_{12} - w_{12}(x_{13}), \quad x_1 -w_1 (x_{13}), \quad x_2 -w_2 (x_{13}), $$
 where $ w_{12}(x_{13})$, $ w_{1}(x_{13})$ and $ w_{2}(x_{13})$ are polynomials of $x_{13}$  with rational coefficients. By substituting the values $ a_{13}$ and  $b_{13}$ for $x_{13}$  into $w_{12}(x_{13})$, $ w_{1}(x_{13})$ and $w_{2}(x_{13})$, we obtain two positive solutions of the system of equations $\{ g_1=0, g_2=0, g_3=0,  g_4=0  \}$ approximately as
 $$( x_{13}, x_{12}, x_1,  x_2)  \approx ( 0.48183112,  0.90692827,  0.20686292, 0.092856189 ), $$  
 $$( x_{13}, x_{12},  x_1,  x_2)  \approx ( 2.7966957,  2.6698577, 0.54677135, 0.28461374 ). 
 $$
We substitute these  values into the system   (\ref{eq431aa}) together with $x_{23} =1$. Then we obtain that 
$r_1 = r_2 =r_{12} = r_{23} = r_{13} \approx   0.47140698$ and  $r_1 = r_2 =r_{12} = r_{23} = r_{13} \approx  0.16491085 $ respectively. We multiply these  solutions by a scale factor and we obtain the two solutions 
 $$( x_1,  x_2, x_{12}, x_{23}, x_{13})  \approx ( 0.097516624, \, 0.043773055, \, 0.42753231, \, 0.47140698, \, 0.22713855 ), $$  
 $$(  x_1,  x_2, x_{12}, x_{23}, x_{13})  \approx ( 0.090168527, \, 0.046935893, \, 0.44028850, \, 0.16491085, \, 0.46120545 ). 
 $$
for the system of equations  
$$
r_1 = r_2 =r_{12} = r_{23} = r_{13}  =1,
$$ 
which {\it are not isometric}.
  Note that due to Proposition \ref{prop5.2}.  these metrics are not naturally reductive. 

Now we consider the case $$({x_{13}}-5) ({x_{13}}-1) \left(7{x_{13}}^2-24 {x_{13}}+14\right)
   \left(287 {x_{13}}^3-625 {x_{13}}^2+369 {x_{13}}-63\right) =0. $$
   
  We   consider ideals $J_1$ generated by 
$\{ g_1, \, g_2, $  $ g_3, g_4,  \,z \,x_1 \, x_2 \, x_{12} \, x_{13} -1, 287 {x_{13}}^3-625 {x_{13}}^2+369 {x_{13}}-63 \}  
$,  $J_2$ generated by 
$\{ g_1, \, g_2, $  $ g_3, g_4,  \,z \,x_1 \, x_2 \, x_{12} \, x_{13} -1,  7{x_{13}}^2-24 {x_{13}}+14  \}$, 
$J_3$ generated by 
$\{ g_1, \, g_2, $  $ g_3, g_4,  \,z \,x_1 \, x_2 \, x_{12} \, x_{13} -1, {x_{13}}-5 \}$ and $J_4$ generated by 
$\{ g_1, \, g_2, $  $ g_3, g_4,  \,z \,x_1 \, x_2 \, x_{12} \, x_{13} -1, {x_{13}}-1 \}$  of the  polynomial ring $R= {\mathbb Q}[z, x_1, x_2, x_{12}, x_{13}] $. 

We take a lexicographic order $>$  with $ z > x_1 >  x_2 > x_{12} > x_{13}$ for a monomial ordering on  $R= {\mathbb Q}[z, x_1, x_2, x_{12}, x_{13}] $. Then, by the aid of computer, we see that a  Gr\"obner basis for the ideal $J_1$ is given by 
\begin{eqnarray*}  & &  \{287 {x_{13}}^3-625 {x_{13}}^2+369 {x_{13}}-63 , x_{12}-1, 117 - 478 x_{13} + 287 {x_{13}}^2 + 42 x_{2}, 
x_{1}-{x_{13}}, \\
& & -123201 + 323882 x_{13} - 173635 {x_{13}}^2 + 378 z \}.    
   \end{eqnarray*}   
  By solving the equation $287 {x_{13}}^3-625 {x_{13}}^2+369 {x_{13}}-63=0$ numerically, we obtain {\it three} positive solutions of the system of equations $\{ g_1=0, g_2=0, g_3=0,  g_4=0, x_{13} -1, 287 {x_{13}}^3-625 {x_{13}}^2+369 {x_{13}}-63=0  \}$.
  Since the solutions satisfy $x_1=x_{13}$,  $x_{12}=x_{23}=1$,
    then Proposition \ref{prop5.2} implies that the metrics  obtained are 
 naturally reductive  with respect to $\SO(8)\times (\SO(5) \times \SO(3))$. 
    
Similarly, we see that a  Gr\"obner basis for the ideal $J_2$ is given by 
\begin{eqnarray*}  & &  \{14 - 24 x_{13 }+ 7 {x_{13}}^2, x_{12} - x_{13}, -1 + x_2, 7 + 7 x_1 - 
 12 x_{13}, -43009 + 15960 x_{13} + 4802 z \}.     
   \end{eqnarray*} 
   By solving the equation $14 - 24 x_{13 }+ 7 {x_{13}}^2=0$, we obtain {\it two} positive solutions of the system of equations $\{ g_1=0, g_2=0, g_3=0,  g_4=0,  14 - 24 x_{13 }+ 7 {x_{13}}^2=0  \}$.
    Since the solutions satisfy $x_{2}=x_{23}=1$, $x_{12}=x_{13}=1$,
    then
  Proposition \ref{prop5.2} implies that the metrics obtained are 
 naturally reductive  with respect to $\SO(8)\times (\SO(4) \times \SO(4))$.  
 (It is possible to check that {\it these two metrics are isometric}).
 
 Similarly, we see that  
a  Gr\"obner basis for the ideal $J_3$ is given by 
\begin{eqnarray*}  & &  \{-5 + x_{13}, -5 + x_{12}, -1 + x_2, -1 + x_1, -1 + 25 z \}  
   \end{eqnarray*}   and we obtain a {\it unique} positive solution of the system of equations $\{ g_1=0, g_2=0, g_3=0,  g_4=0, x_{13} -1, -5 + x_{13}=0  \}$.
   From Proposition \ref{prop5.2} we see that the metric obtained is 
 naturally reductive  with respect to $\SO(8)\times (\SO(4) \times \SO(4))$.

 Finally, we see that 
a  Gr\"obner basis for the ideal $J_4$ is given by 
\begin{eqnarray*}  & &  \{-1 + x_{13}, (-1 + x_{12}) (-5 + 7 x_{12}), -x_{12} + x_2, x_1 - x_{12}, -888 + 
 763 x_{12} + 125 z \},    
   \end{eqnarray*}    and we obtain {\it two} positive solutions of the system of equations $\{ g_1=0, g_2=0, g_3=0,  g_4=0, x_{13} -1=0  \}$. One of the solutions gives the bi-invariant metric $x_1= x_2 = x_{12} = x_{13} = x_{23}=1$ and the other $x_1= x_2 = x_{12} = 5/7,  x_{13} = x_{23}=1$ gives  naturally reductive metric with respect to $\SO(8)\times \SO(7)$ from Proposition \ref{prop5.2}. 
\end{proof}

\smallskip
 \begin{prop}\label{prop3}
For any $n \geq 9$, the Lie group  $\SO(n)$ admits at least one  
left-invariant Einstein metric determined by the
$\Ad( \SO(n-6)\times\SO(3)\times\SO(3))$-invariant scalar product of the form {\em (\ref{metric001})}, which is not naturally reductive.  
 \end{prop}
 \begin{proof}
     We consider the system of equations 
 \begin{equation}\label{eq26} 
r_1 = r_2, \quad r_2 = r_3, \quad r_3 =  r_{12}, \quad  r_{12} = r_{13}, \quad  r_{13} = r_{23}. 
 \end{equation}
 Then finding Einstein metrics of the form (\ref{metric001})  reduces  to finding positive solutions of  system (\ref{eq26}). 
 
 We put $ k_2 = k_3= 3$ and consider  our equations by putting 
 $$x_{12} = x_{13}=1, \quad x_2=x_3.$$  
  
 Then the system of equations (\ref{eq26}) reduces to  
 the system of equations: 
 \begin{equation}\label{eq26b}
\left. { \begin{array}{lll}
g_1&  = & -n {x_1} {x_2}^2 {x_{23}}^2+n {x_2}
   {x_{23}}^2+6 {x_1}^2 {x_2} {x_{23}}^2 \\ & & +6 {x_1} {x_2}^2 {x_{23}}^2-3 {x_1}
   {x_2}^2-{x_1} {x_{23}}^2-8 {x_2}
   {x_{23}}^2 =0, \\
g_2 &= & n {x_1} {x_2} {x_{23}}^2+n
   {x_2}^2 {x_{23}}^2-2 n {x_2} {x_{23}}^2-7
   {x_1} {x_2} {x_{23}}^2 \\ & & -4 {x_2}^2
   {x_{23}}^2+3 {x_2}^2+3 {x_2} {x_{23}}^3+4
   {x_2} {x_{23}}^2+{x_{23}}^2=0, \\
g_3& =& -n {x_1}
   {x_{23}}^2-n {x_{23}}^3+2 n {x_{23}}^2+7
   {x_1} {x_{23}}^2 \\ & & -2 {x_2} {x_{23}}^2+4
   {x_2}+3 {x_{23}}^3-4 {x_{23}}^2-8
   {x_{23}} =0. 
\end{array} } \right\} 
\end{equation}
  We consider a polynomial ring $R= {\mathbb Q}[z,  x_{2},  x_1, x_{23}] $ and an ideal $I$ generated by 
$\{ g_1, \, g_2, \, g_3,$ $ \,z  (\,x_2 - x_{23}) \,x_1 \, x_{23} \, x_{2} -1\}  
$  to find non-zero solutions of equations (\ref{eq26b}) with $ \,x_2 \neq x_{23}$. 
We take a lexicographic order $>$  with $ z >  x_{2} >  x_1 >   x_{23} $ for a monomial ordering on $R$. Then, by the aid of computer, we see that a  Gr\"obner basis for the ideal $I$ contains the  polynomials $\{ h(x_{23}), p_1(x_{23}, x_1), p_2(x_{23}, x_{2})\}$, where $ h(x_{23})$ is a polynomial of   $x_{23}$  given by 
\begin{eqnarray*} 
& & h(x_{23})=(n-6)^2 (n-3) \left(n^2-7 n+24\right)
   {x_{23}}^8-2 (n-6)^2 (n-2)
   \left(n^2-n+6\right) {x_{23}}^7   \\  & &+(n-6) \left(n^4+26
   n^3-269 n^2+686 n-516\right)
   {x_{23}}^6   -44 (n-6) (n-3) (n-2)
   (n+2) {x_{23}}^5   \\  & & +\left(14 n^4+273 n^3-3034 n^2+5687
   n+1164\right) {x_{23}}^4
   -2 (n-2)
   \left(157 n^2-157 n-2778\right)
   {x_{23}}^3  \\  & &+\left(49 n^3+1658 n^2-6539
   n+836\right) {x_{23}}^2  -728 (n-2) (n+5)
  {x_{23}}+2704 (n-1), 
   \end{eqnarray*}
   $p_1(x_{23}, x_1)$ is a polynomial of $x_{23}$ and $x_1$  given by 
 \begin{eqnarray*} 
& &p_1(x_{23}, x_1) = 8 (2 n-5) \left(n^2-7 n+27\right)
   {x_1}+(n-6)^3 (n-3) \left(n^2-7 n+24\right)
   {x_{23}}^7  \\  
   & &-2 (n-6)^3 (n-2)
   \left(n^2-n+6\right) {x_{23}}^6 +(n-6)^2 \left(n^4+19 n^3-199 n^2+371
   n-12\right) {x_{23}}^5 \\ 
   & & -6 (n-6)^2
   (n-2) \left(5 n^2-5 n-58\right) {x_{23}}^4 \\ 
   & & +(n-6) \left(7 n^4+140  n^3-1641 n^2+3090 n+1248\right)
   {x_{23}}^3-104 (n-6)^2 (n-2) (n+5) {x_{23}}^2
  \\ 
   & &  +8
   \left(48 n^3-625 n^2+2305 n-1719\right)
   {x_{23}}   \end{eqnarray*}
 and $ p_2(x_{23}, x_{2})$ is a polynomial of $x_{23}$ and $x_2$  given by 
\begin{eqnarray*} 
& & p_2(x_{23}, x_{2}) = -(n-6)^2 (n-3) \left(n^2-7 n+24\right) \left(2
   n^2-14 n+15\right) {x_{23}}^7 \\ & & +2 (n-6)^2 (n-2)
   \left(n^2-n+6\right) \left(2 n^2-14 n+15\right)
   {x_{23}}^6 \\ & & -(n-6) \left(2 n^6+25 n^5-666
   n^4+3955 n^3-8860 n^2+7452 n-2124\right)
   {x_{23}}^5 \\ & &+2 (n-6) (n-2) \left(31 n^4-248
   n^3+127 n^2+2142 n-2448\right) {x_{23}}^4 \\ & &-\left(15 n^6+194 n^5-5442
   n^4+33531 n^3-73361 n^2+38979 n+18396\right)
   {x_{23}}^3 \\ & &  +2
   (n-2) \left(119 n^4-952 n^3-1276 n^2+21873
   n-28098\right) {x_{23}}^2 \\
   & & -\left(7 n^5+849
   n^4-11830 n^3+53569 n^2-79135 n+24552\right)
   {x_{23}} \\  & & +52 (n-7) (n-1) \left(n^2-7
   n+27\right)  {x_2} +624 (n-2) \left(n^2-7
   n+27\right). 
    \end{eqnarray*}
  Thus we see that,  if there exists a real root  $x_{23}= \alpha_{23}$ of $h(x_{23}) = 0$, then 
  there are a real solution $x_1= \alpha_1$ of $p_1(\alpha_{23}, x_1) =0$ and  a real solution $x_2= \alpha_2$ of $p_2(\alpha_{23}, x_2) =0$. 
    
 Now we have $h(0) =2704 (n-1)>0$ for $n > 1$,    
  $h(2) = 4 (16 n^5-424 n^4+4625 n^3-25470 n^2+70193
   n-77128 ) $  $= 4 (16 (n-6)^5$ $+56 (n-6)^4+209 (n-6)^3+756 (n-6)^2+1397
   (n-6)+1022 ) > 0$ for $n \geq 6$ and  $h(1) = -2 (n-9) (n-1) n^2 < 0 $ for $n > 9$. 
   Note that for $ n = 9$ $h(6/5)  = - 1751152/390625 < 0$. 
   Thus we see that the equation $h(x_{23}) = 0$ has {\it two} positive roots $x_{23}= \alpha_{23}, \beta_{23}$ with $ 0 < \alpha_{23} < 1 < \beta_{23} < 2$ for $n  > 9$. For $n= 9$ we have  roots  $x_{23}= 1, \beta_{23}$  with  $ 6/5  < \beta_{23} < 2$. 
   
   Let $\gamma =\alpha _{23}$ or $\beta _{23}$.
    We have to show that the real solutions $x_1= \alpha_1$ of $p_1(\gamma, x_1) =0$ and
   $x_2= \alpha_2$ of $p_2(\gamma, x_2) =0$ are positive.
   To this end, we take a lexicographic order $>$  with $ z >  x_{2} >  x_{23} >   x_{1} $ for a monomial ordering on $R$. Then, by the aid of computer, we see that a  Gr\"obner basis for the ideal $I$ contains the  polynomial $h_{1}(x_1) $ of $x_1$ given by 
{\small   \begin{eqnarray*} 
& & h_{1}(x_1) = 4 (n-1) \left(n^2-7 n+24\right) \left(n^2-7
   n+27\right)^2 {x_{1}}^8 -16
   (n-2) \left(n^2-7 n+27\right)\times  \\ & & (2 n^4-28
   n^3+170 n^2-504 n+549 ) {x_{1}}^7 
   +(112 n^7-2728 n^6+29992
   n^5-192017 n^4  \\ & &+761574 n^3  -1849727 n^2+2498826
   n-1434888) {x_{1}}^6 -4 (n-2) (56 n^6-1348 n^5+13792
   n^4  \\ & &-77805 n^3+254449 n^2-453225 n+344070)
   {x_{1}}^5    +(280 n^7-7880
   n^6+93373 n^5-609014 n^4 \\ & &+2369548 n^3-5476199
   n^2+6921202 n-3695208)
   {x_{1}}^4 \\ & & 
   -2 (n-8) (n-2) \left(112 n^5-2304 n^4+18506
   n^3-73480 n^2+144545 n-109506\right)
   {x_{1}}^3 \\ & &+2 (n-8)^2 \left(56 n^5-972 n^4+6475
   n^3-20866 n^2+32361 n-18921\right)
   {x_{1}}^2 \\ & & -2 (n-8)^3 (n-2) (2 n-5) \left(8 n^2-64 n+117\right) {x_{1}}+(n-8)^4 (n-3) (2 n-5)^2.     \end{eqnarray*}
   }
     Now we have  
   {\small   \begin{eqnarray*} 
& & h_{1}(x_1) =   4 (n-1) \left(n^2-7 n+24\right) \left(n^2-7
   n+27\right)^2 {x_{1}}^8 -16
   (n-2) \left(n^2-7 n+27\right)\times  \\ & & (2 (n-8)^4+36 (n-8)^3+266 (n-8)^2+936 (n-8)+1253) {x_{1}}^7  
   +(112 (n-8)^7 \\ & &
   +3544 (n-8)^6+49576 (n-8)^5+395823
   (n-8)^4+1933510 (n-8)^3  +5714577 (n-8)^2 \\ & & +9285018
   (n-8)+6127496) {x_{1}}^6 
   -4 (n-2) (56 (n-8)^6+1340 (n-8)^5 +13632 (n-8)^4 \\ & & +74259
   (n-8)^3+222137 (n-8)^2  +328423 (n-8)+167678)
   {x_{1}}^5    +(280 (n-8)^7  \\ & &+7800 (n-8)^6+91453 (n-8)^5  +578706
   (n-8)^4+2089420 (n-8)^3+4129977 (n-8)^2  \\ & &
   +3804802  (n-8)+1039208)
   {x_{1}}^4  -2 (n-8) (n-2) (112 (n-8)^5+2176 (n-8)^4  \\ & &
   +16458 (n-8)^3+59368
   (n-8)^2+97185 (n-8)+5203)
   {x_{1}}^3  +2 (n-8)^2 (56 (n-8)^5  \\ & &+1268 (n-8)^4+11211 (n-8)^3+48006
   (n-8)^2+97929 (n-8)+73439)
   {x_{1}}^2 \\ & & -2 (n-8)^3 (n-2) (2 n-5) (8 (n-8)^2+64 (n-8)+117) {x_{1}}+(n-8)^4 (n-3) (2 n-5)^2.     \end{eqnarray*}
   }
    Thus we see that,  for $n \geq 9$, the coefficients of the polynomial  $h_{1}(x_{1})$ are positive  for  even degree terms and negative for odd degree terms, so  if  the equation $h_{1}(x_{1})  = 0$ has real solutions then these are all  positive.

   We also take a lexicographic order $>$  with $ z >  x_{1} >  x_{23} >   x_{2} $ for a monomial ordering on $R$. Then, by the aid of computer, we see that a  Gr\"obner basis for the ideal $I$ contains the  polynomial $h_{2}(x_2) $ of $x_2$ given by 
{\small \begin{eqnarray*} 
& & h_{2}(x_2)= 64 (n-6)^2 (2 n-5)^2 \left(n^2-7 n+27\right)^2
   {x_{2}}^8 \\ & & -896 (n-6)^2 (n-2) (2 n-5)
   \left(n^2-7 n+27\right) \left(n^2-n+12\right)
   {x_{2}}^7 \\
   & &
   +4 (n-6) (1176 n^7-14228
   n^6+100368 n^5-730649 n^4+4440678 n^3 \\ & & -18369941
   n^2+41390868 n-33209244)
   {x_{2}}^6 
   -8 (n-6) (n-2) (686 n^6-2200 n^5 \\ & &+33593 n^4 -489642
   n^3+2433897 n^2  -7853838 n+19276848)
   {x_{2}}^5 \\
   & & +(2401 n^8-2114 n^7+85477
   n^6-3433940 n^5+22264067 n^4 \\ & & -66085822
   n^3+304096111 n^2-1233542964
   n+1558955520) {x_{2}}^4 \\ & & 
   -2 (n-2) (4949 n^6-15874
   n^5+114730 n^4-3099532 n^3 \\ & &+11930753
   n^2+23543310 n-121637952) {x_{2}}^3
   \\ & & +3
   \left(5099 n^6-34656 n^5+78010 n^4-1041692
   n^3+8171395 n^2-21025340 n+15585312\right)
   {x_{2}}^2 \\ & &-104 (n-2) \left(101 n^4-661 n^3+743
   n^2+5001 n-15768\right) {x_{2}} \\ & &+2704 (n-3) (n-1) \left(n^2-7 n+24\right)   \end{eqnarray*}
   }

   Now we have 
   {\small \begin{eqnarray*} 
& & h_{2}(x_2)= 64 (n-6)^2 (2 n-5)^2 \left(n^2-7 n+27\right)^2
   {x_{2}}^8  -896 (n-6)^2 (n-2) (2 n-5)\times \\& &
   \left(n^2-7 n+27\right) \left(n^2-n+12\right)
   {x_{2}}^7    +4 (n-6) (1176 (n - 8)^7 + 51628 (n - 8)^6 + 997968
     (n - 8)^5 \\& &  + 10699111 (n - 8)^4 + 68192070 (n - 8)^3 + 256591483
     (n - 8)^2 + 519880260 (n - 8)  \\ & &+ 428454852)
   {x_{2}}^6 
   -8 (n-6) (n-2) (686 (n - 8)^6 + 30728 (n - 8)^5 + 604153 (n - 8)^4\\ & & + 6201974
     (n - 8)^3 + 34466041 (n - 8)^2 + 95692802
     (n - 8) + 106856960)
   {x_{2}}^5 \\
   & & +(2401 (n - 8)^8 + 151550 (n - 8)^7 + 4269685
     (n - 8)^6 + 66669212 (n - 8)^5 + 617496227
     (n - 8)^4  \\ & &+ 3426718370 (n - 8)^3 + 11106086431
     (n - 8)^2 + 19421585724 (n - 8) + 14243093536) {x_{2}}^4 
     \\ & & 
   -2 (n-2) (4949
     (n - 8)^6 + 221678 (n - 8)^5 + 4230810 (n - 8)^4 + 41090228
     (n - 8)^3\\ & & + 204389985 (n - 8)^2 + 502205726
     (n - 8) + 490441840) {x_{2}}^3
    +3
(5099 (n - 8)^6 + 210096 (n - 8)^5 \\ && + 3586810
     (n - 8)^4 + 31488548 (n - 8)^3 + 148970467
     (n - 8)^2 + 362225908 (n - 8) \\ & &+ 357598976)
   {x_{2}}^2 -104 (n-2) (101 (n - 8)^4 + 2571 (n - 8)^3 + 23663 (n - 8)^2 + 96825
     (n - 8) \\ & &+ 147056) {x_{2}} +2704 (n-3) (n-1) \left(n^2-7 n+24\right).    \end{eqnarray*}
   } 
  Thus we see that,  for $n \geq 8$, the coefficients of the polynomial  $h_{2}(x_{2})$ are positive  for  even degree terms and negative for odd degree terms and that, so if  the equation $h_{2}(x_{2})  = 0$ has real solutions then these are all  positive. 
  Since the solutions satisfy the property $x_2\ne x_{23}$, $x_{23}\ne 1$ and $x_{12}= x_{13}=1$,
    then
  Proposition \ref{prop5.1} implies that the metrics obtained are not
 naturally reductive 
  \end{proof}

From the above Propositions \ref{prop1}, \ref{prop2} and \ref{prop3} we obtain Theorem \ref{main}.

\section{The Lie groups $\SO((n-2)+1+1)$}

In the present section we consider
the scalar products (\ref{metric003}) on the 
  Lie groups
$\SO(k_1+k_2+k_3)$, $k_1=n-2, k_2=k_3=1$, and
prove that for $n\ge 5$ we obtain only naturally reductive Einstein metrics.
In this case decomposition (\ref{decom_so(n)}) becomes
$$
\mathfrak{so}((n-2)+1+1)=\mathfrak{m}_1\oplus \mathfrak{m}_{12}\oplus\mathfrak{m}_{13}\oplus \mathfrak{m}_{23},
$$
where $\mathfrak{m}_{12}$, $\mathfrak{m}_{13}$ are equivalent as $\Ad(\SO(n-2))$-modules.

\begin{lemma}  For the metric {\em(\ref{metric003})} it is
$$
r(\mathfrak{m}_{12}, \mathfrak{m}_{13})=(0).
$$
\end{lemma}
\begin{proof} We use the formula for the Ricci curvature in \cite[Corollary 7.38, p. 184]{Be} adapted  for the case of a
Lie group.  Then using polarization we  obtain that
\begin{equation}\label{riccipolar}
r(X, Y)=-\frac12\sum _j\langle [X, X_j], [Y, X_j]\rangle-\frac12 B(X, Y)+\frac14 \sum _{i, j}
\langle [X_i, X_j], X\rangle \langle [X_i, X_j], Y\rangle,
\end{equation}
where $\{X_i\}$ is an orthonormal  basis of $\fr{so}(n)$ with respect to the metric (\ref{metric003}).
For any $X\in\fr{m}_{12}, Y\in\fr{m}_{13}$ we need to show that $r(X, Y)=0$.  A computation using (\ref{modules}) shows that 
$\fr{m}_{12}, \fr{m}_{13}$ are orthogonal with respect to $-B$, so the second term in (\ref{riccipolar}) vanishes.

We claim that (i) $\langle [X, X_j], [Y, X_j]\rangle=0$ and 
(ii) $\langle [X_i, X_j], X\rangle \langle [X_i, X_j], Y\rangle=0$.
Indeed,  using the orthonormal basis $\{e_{ab}: 1\le a<b\le n\}$ of $\fr{so}(n)$ introduced in Section 3  we let
$X=e_{i, n-1}\ (1\le i\le n-2)$, $Y=e_{jn}\ (1\le j\le n-2)$.

To prove claim (i) let $X_j=e_{ab}\ (a<b)$.  By using Lemma \ref{brac} it follows that
$$
[X, X_j]=[e_{i, n-1}, e_{ab}]=
\begin{cases}
\pm e_{ib},\quad \quad\  a=n-1\\
\pm e_{b, n-1},\quad a=i\\
\pm e_{ia},\quad \quad\  b=n-1\\
\pm e_{a, n-1},\quad b=i
\end{cases}
$$
and
$$
[Y, X_j]=[e_{j n}, e_{ab}]=
\begin{cases}
\pm e_{jb},\quad a=n\\
\pm e_{bn},\quad a=j\\
\pm e_{an},\quad b=j\\
\pm e_{ja},\quad b=n.
\end{cases}
$$
By using these  expressions for the Lie brackets we can check that
$-B([X, X_j], [Y, X_j])=0$, hence for the metric (\ref{metric003}) it is also
$\langle [X, X_j], [Y, X_j]\rangle=0$.
To prove claim (ii) we let $X_i=e_{ab}\ (a<b)$, $X_j=e_{cd}\ (c<d)$.  By using
Lemma \ref{brac} and obtain that
$$
[X_i, X_j]=[e_{ab}, e_{cd}]=
\begin{cases}
\pm e_{ad},\quad b=c\\
\pm e_{ac},\quad b=d\\
\pm e_{bd},\quad a=c\\
\pm e_{bc},\quad a=d.
\end{cases}
$$
The conclusion then follows by a similar argument as for the proof of claim (i).
\end{proof}

\medskip
 \begin{prop}\label{prop4}
For any $n \geq 5$,  the only 
left-invariant Einstein metrics on the Lie group  $\SO(n)$, determined by the
$\Ad( \SO(n-2))$-invariant scalar products of the form {\em(\ref{metric003})},  are naturally reductive
metrics. 
 \end{prop}
 \begin{proof}
 The proof involves  manipulations of polynomials using Gr\"obner bases, and it is quite extensive   to be presented in its complete form.  Since the metrics obtained are only naturally reductive, we will only sketch the ideas  behind these computations.
 
 We use the components for the Ricci tensor in Proposition \ref{prop5.5} and 
  consider the system of equations   
 \begin{equation}\label{eq444} 
 r_1 =  r_{12}, \, \, r_{12} = r_{23}, \,\,  r_{23} = r_{13}. 
 \end{equation}
Then finding Einstein metrics of the form (\ref{metric003})  reduces  to finding the positive solutions of system (\ref{eq444}),  and  we normalize  our equations by putting $x_{23}=1$. Then we have the system of equations: 
 \begin{equation}\label{eq444_a} 
\left. { \begin{array}{l}
g_1=  n {x_1}^2 {x_{13}}^2-2 n {x_1} {x_{12}} {x_{13}}^2+n {x_{12}}^2 {x_{13}}^2+{x_1}^2 {x_{12}}^2-2 {x_1}^2 {x_{13}}^2-{x_1} {x_{12}}^3 {x_{13}}\\
 \quad  \quad  +{x_1} {x_{12}} {x_{13}}^3 +4 {x_1} {x_{12}} {x_{13}}^2+{x_1} {x_{12}} {x_{13}}-4 {x_{12}}^2 {x_{13}}^2=0,  \\ 
g_2=  -n {x_1} {x_{13}}+n {x_{12}}^3-2 n {x_{12}}^2 {x_{13}}+n {x_{12}} {x_{13}}^2+2 n {x_{12}} {x_{13}}-n {x_{12}}+3 {x_1} {x_{13}}\\
  \quad  \quad -{x_{12}}^3+4 {x_{12}}^2 {x_{13}}-3 {x_{12}} {x_{13}}^2-4 {x_{12}} {x_{13}}+{x_{12}}=0,  \\ 
g_3=  n {x_1} {x_{12}}-n {x_{12}}^2 {x_{13}}+2 n {x_{12}} {x_{13}}^2-2 n {x_{12}} {x_{13}}-n {x_{13}}^3+n {x_{13}}-3 {x_1} {x_{12}}\\
   \quad  \quad +3 {x_{12}}^2 {x_{13}}-4 {x_{12}} {x_{13}}^2+4 {x_{12}} {x_{13}}+{x_{13}}^3-{x_{13}}=0. 
\end{array} } \right\} 
\end{equation}

  We consider a polynomial ring $R= {\mathbb Q}[z, x_1, x_{12}, x_{13}] $ and an ideal $I$ generated by 
$\{ g_1, \, g_2, $  $ g_3,   \,z \,x_1 \, x_{12} \, x_{13} -1\}  
$  to find non zero solutions of equations (\ref{eq444_a}). 
We take a lexicographic order $>$  with $ z >  x_{13}> x_{12} > x_1$ for a monomial ordering on $R$. Then, by the aid of computer, we see that a  Gr\"obner basis for the ideal $I$ contains the  polynomial

\begin{equation}\label{product}
({x_{1}}-1) \left((n-1) {x_{1}}-(n-3) \right)\left(\left(n^3-2 n^2+n-4\right){x_1}-(n-4) (n-1)^2\right) A(x_1),
\end{equation}
where
\begin{equation*}\label{A} 
 \begin{array}{l}
A(x_1) = 16 (n-3)^2 (n-2)^3 (n-1) {x_1}^3-4 (n-3) (n-2)^2 \left(5 n^3-18 n^2+21 n-4\right) {x_1}^2\\
 \qquad\qquad +4 (n-2)^2 (n-1) \left(n^4-27 n^2+70 n-32\right) {x_1}\\
 \qquad\qquad -(n-4) (n-3) (n-1)^2 n \left(n^3-6 n^2+25 n-32\right)
\end{array} 
\end{equation*}
is a polynomial in $x_1$ of degree $3$.
We divide our study in the following cases:

\medskip
\noindent
\underline{Case (a)} $A(x_1)=0$.  We claim that the system of equations 
(\ref{eq444_a}) has no real solutions in this case.

\noindent
We consider the ideal $J$ generated by $\{g_1, g_2, g_3, A(x_1), \,z \,x_1 \, x_{12} \, x_{13} -1\}$ of the polynomial ring $R= {\mathbb Q}[z, x_1, x_{12}, x_{13}]$.
We take a lexicographic order $>$  with $ z >  x_{13}> x_{12} > x_1$  for a monomial ordering on $R$. Then, by the aid of computer, we see that a  Gr\"obner basis for the ideal $J$ contains the  polynomial

\begin{equation*}
\small{
  \begin{array}{l}
 p(x_1, x_{12}) = 4 (n-2)^2 (n-1) \left(n^4-10 n^3+37 n^2-32 n+16\right){x_{12}}^2  
-2 (n-2)\times \\ \left(\left(n^4+2 n^3-3 n^2-36 n+64\right) (n-1)^2 \right.  
   +2 (n-2) \left(2 n^5-25 n^4+93 n^3-143 n^2 \right.
 \left.  +93 n-28\right){x_1}\\ + \left. 8 (n-3) (n-2)^2
   (n+1) (n-1) {x_1}^2\right){x_{12}}   
   + (n-1) \left(4 (n-3) (n-2)^2 (n-1) \left(n^2-3 n+4\right) \right.
   {x_1}^2  \\ -8 (n-2) \left(3 n^4-17 n^3+31 n^2-21 n+8\right)
   {x_1}   \left.+(n-3) (n-1) \left(n^3-6 n^2+25 n-32\right)
   n^2\right).
   \end{array}
   }
\end{equation*}

We view the above polynomial as a polynomial of the variable $x_{12}$, that is
$$
p(x_1, x_{12})= \tilde{p}(x_{12})=4 (n-2)^2 (n-1)f_2(n){x_{12}}^2+f_1(x_1, n)x_{12}+f_0(x_1, n),
$$
where $f_2(n)=(n-4)^4+ 6(n-4)^3 +13(n-4)^2 +40 (n-4) +96>0$ for $n\ge 4$. 
We will show that, for the roots of the equation $A(x_1)=0$, the polynomial $\tilde{p}(x_{12})$ has no real roots for $x_{12}$.
Indeed, the discriminant $D_{\tilde{p}(x_{12})}$ of the polynomial $\tilde{p}(x_{12})$ has the form
$$
D_{\tilde{p}(x_{12})}=f_1^2-16(n-2)^2(n-1)f_2f_0=-4(n-2)^2\cdot \tilde{q}(x_1),
$$
where
$\tilde{q}(x_1)$
is a polynomial in $x_1$ of degree $4$. 
We need to show that  $D_{\tilde{p}(x_{12})}<0$  whenever $x_1=\alpha$, with $A(\alpha)=0$ and $\alpha >0$ (since we
are interested to Riemannian metrics).
By dividing the polynomial $\tilde{q}(x_1)$ by $A(x_1)$ we obtain that
$$
\tilde{q}(x_1)=A(x_1)B(x_1)+q(x_1),
$$
where $q(x_1)$ is a polynomial of degree $2$ given by
$$
q(x_1)=4(n-1)f_2(n)(a_0(n)+a_1(n)x_1+a_2(n){x_1}^2).
$$
Then we  need to show that $q(x_1)>0$  or that 
$$
r(x_1)\equiv a_0(n)+a_1(n)x_1+a_2(n){x_1}^2>0,
$$ 
where the coefficients of the polynomial $r(x_1)$ are explicitly given as
\begin{eqnarray*}
a_0(n) &=& (n-1)(5 n^5-38 n^4+113 n^3-196 n^2+192 n-64),\\
  a_1(n) & =& (n-2)(n^6-15 n^5+69 n^4-141 n^3+150 n^2-104 n+32),\\
  a_2(n) &= &4 (n-3) (n-2)^2 (n-1)(n^2-3 n+4)>0\ \ \mbox{for}\  n\ge 4.
\end{eqnarray*}
If $n\ge 9$ we expand the  polynomials $a_0(n), a_1(n)$ as

\begin{eqnarray*}
a_0(n) &=& 5 (n-9)^6+227 (n-9)^5+4291 (n-9)^4+43197 (n-9)^3\\
 & &\   +244036 (n-9)^2+732812 (n-9)+912736,\\
a_1(n) &=& (n-9)^7+46 (n-9)^6+882 (n-9)^5+9036 (n-9)^4+52353 (n-9)^3\\
  & &\   +164350 (n-9)^2+229164 (n-9)+48104,\\
\end{eqnarray*}
hence it follows that
$r(x_1)>0$.

\noindent
\smallskip
If $5\le n <9$ we consider the  discriminant $D_{r(x_1)}$ of the polynomial $r(x_1)$ and this is given by
\begin{eqnarray*}
D_{r(x_1)} &=&
(n-2)^3  (n^{11}-28 n^{10}+227 n^9-650 n^8-761 n^7+11240 n^6-38635 n^5\\
&&\  +75262 n^4-93472 n^3+72352 n^2-31232 n+5632).
\end{eqnarray*}
It is easy to see that for $n=5, 6, 7, 8$ it is $D_{r(x_1)}<0$ so $r(x_1)>0$. Therefore
$D_{\tilde{p}(x_{12})}<0$  for all $n \geq 5$ and this completes the proof of Case (a).

\medskip
\noindent
\underline{Case (b)} $A(x_1)\ne0$.  In this case we obtain only naturally reductive Einstein metrics.

\noindent
Indeed,   from equation (\ref{product}) we obtain 
 the solutions
$$
x_{1}=1,\qquad
x_1=\frac{n-3}{n-1},\qquad
x_1=\frac{(n-4)(n-1)^2}{n^3-2n^2+n-4}.
$$
By  substituting  the above solutions to the system (\ref{eq444_a}) and 
computing    Gr\"obner bases for this system,  we obtain the  solutions
$$
(x_1, x_{12}, x_{13}, x_{23})=
(1, 1, 1, 1),
$$
$$
(x_1, x_{12}, x_{13}, x_{23})=
\left(\frac{n-3}{n-1}, 1, \frac{n-3}{n-1}, 1\right),
$$
$$
(x_1, x_{12}, x_{13}, x_{23})=
\left(\frac{(n-4)(n-1)^2}{n^3-2n^2+n-4}, \frac{(n-1)(n^2-3n+4)}{n^3-2n^2+n-4}, 
\frac{(n-1)(n^2-3n+4)}{n^3-2n^2+n-4}, 1\right).
$$
From Proposition \ref{prop5.33} it follows that the above metrics are naturally reductive
with respect to $\SO(n)\times L$, where $L$ is $\SO(n)$, $\SO(n-1)$ or $\SO(n-2)\times\SO(2)$ respectively.
 \end{proof}

\bigskip
By working in a similar manner as in the above proofs, we can obtain Table 1
for the Lie groups $\SO(n)=\SO(k_1+k_2+k_3)$, which lists the numbers of the
   new non naturally reductive and naturally reductive left-invariant 
Einstein metrics  of the forms (\ref{metric001}),  (\ref{metric002}) or (\ref{metric003}),  up to isometry.
The table also contains results  that due to space limitations we did not provide explicit calculations.

\medskip
  \begin{center}
\small{
\begin{table}
 \begin{tabular}{|c|c|c|c|}
  \hline 
 $\SO(k_1+k_2+k_3)$  & $(k_1, k_2, k_3)$ & Non-naturally reductive &  naturally reductive\\ 
 \hline 
 $\SO(5)$    & $(3, 1, 1)$ & $0$ & $3$  \\
 \hline 
     & $(2, 2, 1)$ & $0$ & $3$  \\
     \hline 
 $\SO(6)$    & $(4, 1, 1)$ & $0$ & $3$  \\
       \hline 
     & $(3, 2, 1)$ & $0$ & $5$  \\
       \hline 
   & $(2, 2, 2)$ & $0$ & $2$ \\
       \hline 
 $\SO(7)$    & $(5, 1, 1)$ & $0$ & $3$ \\
       \hline 
          & $(4, 2, 1)$ & $0$ & $6$ \\
       \hline 
        & $(3, 3, 1)$ & ${\bf 1}$ & $5$ \\
       \hline 
       & $(3, 2, 2)$ & $0$ & $5$ \\
       \hline 
 $\SO(8)$    & $(6, 1, 1)$ & $0$ & $3$ \\
       \hline 
        & $(5, 2, 1)$ & $0$ & $6$ \\
       \hline 
        & $(4, 3, 1)$ & ${\bf 2}$ & $7$ \\
       \hline        
   & $(4, 2, 2)$ & $0$ & $5$ \\
       \hline  
       & $(3, 3, 2)$ & ${\bf 1}$ & $5$ \\
       \hline     
 $\SO(9)$    & $(7, 1, 1)$ & $0$ & $3$ \\
       \hline 
        & $(6, 2, 1)$ & $0$ & $6$ \\
       \hline 
        & $(5, 3, 1)$ & ${\bf 2}$ & $8$ \\
       \hline        
   & $(5, 2, 2)$ & $0$ & $5$ \\
       \hline  
       & $(4, 3, 2)$ & ${\bf 2}$ & $8$ \\
       \hline  
  & $(4, 4, 1)$ & ${\bf 2}$ & $5$ \\
       \hline  
       & $(3, 3, 3)$ & ${\bf 2}$ & $5$ \\
       \hline                      
 \end{tabular}
 \bigskip
\caption{Numbers of  new non naturally reductive and naturally reductive left-invariant 
Einstein metrics on the Lie group $\SO(n)=\SO(k_1+k_2+k_3)$  up to isometry.  These are $\Ad(\SO(k_1)\times\SO(k_2)\times\SO(k_3))$-invariant metrics  of the forms (\ref{metric001}), (\ref{metric002}) or (\ref{metric003}).}\label{table}
\end{table}
}
 \end{center}

\end{document}